\documentclass[11pt]{amsart}
\usepackage{amssymb,amsmath,amsthm}
\usepackage{epstopdf}
\usepackage{mdframed}
\usepackage{empheq}
\usepackage{bm}
\usepackage{hyperref}
\usepackage{enumitem}
\usepackage{subcaption}
\usepackage{todonotes}
\usepackage[round]{natbib}
\allowdisplaybreaks[4]
\setcitestyle{numbers,square}

\newtheorem{theorem}{Theorem}
\newtheorem{definition}[theorem]{Definition}

\newtheorem{lemma}[theorem]{Lemma}

\newtheorem{remark}[theorem]{Remark}

\newcommand\uX{\underline X}

\newcommand\bx{\underline x}
\newcommand\ua{\underline \alpha}

\newcommand\cA{\mathcal A}

\newcommand\cC{\mathcal C}

\newcommand\cF{\mathcal F}

\newcommand\cI{\mathcal I}
\newcommand\cL{\mathcal L}

\newcommand\cP{\mathcal P}

\newcommand\cX{\mathcal X}
\newcommand\cY{\mathcal Y}
\newcommand\cZ{\mathcal Z}
\newcommand\cW{\mathcal W}

\def\AA{\mathbb{A}}
\newcommand\EE{\mathbb E}

\newcommand\PP{\mathbb P}
\newcommand\NN{\mathbb N}
\newcommand\RR{\mathbb R}

\def\AA{\mathbb{A}}

\newcommand{\x}{\bx}

\newcommand{\xdim}{\ell}

\makeatletter
\renewcommand\l@paragraph[2]{}
\renewcommand\l@subparagraph[2]{}
\makeatother
\DeclareMathOperator*{\argmin}{arg\,min}

\title[LDP for large population games]{Laplace principle for large population games with control interaction}

\author{Peng Luo \and Ludovic Tangpi}
\date{\today}
	\thanks{}

\begin{document}

\begin{abstract}
	This work investigates continuous time stochastic differential games with a large number of players whose costs and dynamics interact through the empirical distribution of both their states and their controls.
	The control processes are assumed to be open-loop.
	We give regularity conditions guaranteeing that if the finite-player game admits a Nash equilibrium, then both the sequence of equilibria and the corresponding state processes satisfy a Sanov-type large deviation principle.
	The results require existence of a Lipschitz continuous solution of the master equation of the corresponding mean field game, and they carry over to cooperative (i.e. central planner) games.
	We study a linear-quadratic case of such games in details.
\end{abstract}

\maketitle


\section{Introduction}

This paper is a sequel to \cite{pontryagin} in which the convergence of symmetric, continuous time stochastic differential games to mean field games was analyzed.
Here the goal is to complement the convergence results by deriving large deviation principles (in Laplace form) for the sequence of Nash equilibrium and the associated state processes. 
Let us briefly describe the stochastic differential game we consider, in its \emph{non-cooperative} version.
The cooperative case is discussed in the last section of the article.
Let $d \in \mathbb{N}$ be fixed and let $(\Omega,\mathcal{F}, \PP)$ be a probability space which is rich enough to carry a (fixed) sequence of $d$-dimensional independent Brownian motions $(W^i)_{i\ge1}$.
Given $N \in \mathbb{N}$ and $T>0$, consider the filtration $\mathbb{F}^N:=(\mathcal{F}_t^N)_{t\in[0,T]}$, which is the $\PP$--completion of $\sigma(W^i_s, s\le t, i=1,\dots, N)$.
Without further mention, we will always use the identification
\begin{equation*}
	W \equiv W^1, \quad  \cF_t \equiv \cF^1_t\quad  \text{and}\quad \mathbb{F} \equiv \mathbb{F}^1.
\end{equation*}
Given two functions $f$ and $g$, the cost that agent $i$ seeks to minimize, when the strategy profile of the $N$ players is $\underline \alpha := (\alpha^1,\dots,\alpha^N)$, is
$$
	J(\alpha^i; \underline \alpha^{-i}):=
	J^i(\underline \alpha) :=
	\EE \bigg[ \int_0^T f(t,X^{i, \underline \alpha}_t, \alpha^i_t, L^{N}(\underline X^{\underline \alpha}_t, \underline \alpha_t)) dt + g(X^{i, \underline \alpha}_T, L^{N}(\underline X^{\underline\alpha}_T) )\bigg],
$$
where we denote 
$$\underline x^{-i}:= (x^1,\dots,x^{i-1}, x^{i+1},\dots, x^N)\quad \text{and} \quad \x := (x^1,\dots, x^N),$$
and where the position $X^{i,\ua}$ of player $i$ is given by the controlled diffusion
\begin{equation}
\label{eq:N-SDE}
	dX^{i, \underline \alpha}_t = b\big(t,X^{i, \underline \alpha}_t, \alpha^i_t, L^{N}(\underline X_t^{\underline\alpha}, \underline \alpha_t)\big) dt + \sigma dW^i_t, \quad X^{i, \underline\alpha}_0 =x \in \mathbb{R}^\ell,
\end{equation}
$\ell \in \mathbb{N}$.
The term $L^{N}(\underline X_t^{\underline\alpha}, \underline \alpha_t)$ models the \emph{interaction} between the players.
We assume that the players are in weak interaction through the empirical distribution of both the states and the strategies of the whole system.
That is, we let
\begin{equation*}
	L^N(\x) := \frac1N\sum_{i=1}^N\delta_{x^i}.
\end{equation*}
It is interesting to notice that in this model, the players interact not only through their respective states (or positions) $X^{i,\underline\alpha}$, but through their \emph{controls} $\alpha^{i}$ as well.
The admissible set $\cA$ of the controls is defined as follows: given $m \in \NN$ and a closed convex set $\mathbb{A}\subseteq \mathbb{R}^m$, we let
	\begin{equation*}
		\mathcal{A}:= \bigg\{\alpha:[0,T]\times \Omega\to \mathbb{A},\,\, \mathbb{F}^N\text{-progressive such that }  \EE\bigg[\int_0^T|\alpha_t|^2\,dt \bigg]<\infty \bigg\}.
	\end{equation*}
As usual, one is interested in a Nash equilibrium $\hat{\underline \alpha}:= (\hat\alpha^{i},\dots, \hat\alpha^{N})$.
That is, admissible strategies $(\hat\alpha^{1},\dots,\hat\alpha^N)$ such that for every $i=1,\dots,N$ and $\alpha \in \mathcal{A}$ it holds that
\begin{equation*}
 	J^i(\hat{\underline\alpha}) \le J(\alpha;\hat{\underline\alpha}^{-i}).
\end{equation*}

Denoting by $\mathcal{P}_2(\mathbb{R}^\xdim \times\mathbb{R}^m)$ (respectively $\mathcal{P}_2(\mathbb{R}^\xdim)$) the set of probability measures with finite second moment on $\mathbb{R}^\xdim \times \mathbb{R}^m$ (respectively $\mathbb{R}^\xdim$), one formally associates the above $N$-player game to the following mean field game:
Given a flow of distributions $(\xi_t)_{t \in [0,T]}$ with $\xi_t \in \cP_2(\mathbb{R}^\xdim\times \mathbb{R}^m)$ with first marginal $\mu_t\in \cP_2(\mathbb{R}^\xdim)$, consider a solution $\hat\alpha^\xi$ of the control problem
\begin{equation*}
	\begin{cases}
		\inf_{\alpha \in \mathfrak{A}}\EE\left[\int_0^T f(t,X^\alpha_t, \alpha_t,\xi_t) dt + g(X^\alpha_T,\mu_T) \right]\\
		dX^\alpha_t = b(t,X^\alpha_t, \alpha_t, \xi_t) dt + \sigma dW_t, \qquad X^\alpha_0 =x.
	\end{cases}
\end{equation*}
Hereby, $\mathfrak A$ is the set of admissible controls for the mean field game is defined as
\begin{equation*}
	\mathfrak{A}:= \bigg\{\alpha:[0,T]\times \Omega\to \mathbb{A},\,\, \mathbb{F}\text{-progressive such that }  \EE\bigg[\int_0^T|\alpha_t|^2\,dt \bigg]<\infty \bigg\}.
\end{equation*}
A mean field equilibrium is a strategy $\hat\alpha^\xi \in \mathfrak A$ satisfying the fixed point (or consistency) condition
\begin{equation*}
	\xi_t = \cL(X^{\hat\alpha^\xi}_t, \hat\alpha^\xi_t)\quad \text{for all } t.
\end{equation*}
Such a mean field game is usually referred to as "mean field game of controls" or "extended mean field game" and has been introduced by \citet{MR3160525,Gomes-Mohr-Sou13}.
The interest in this type of mean field games quickly blossomed, mainly due to their natural applications in economics and finance, see e.g. \cite{MR3805247,MR3359708,MR3755719} or \cite[Section 3.3.1]{MR3195844} and energy production models \cite{MR3575612,alasseur2017extended}.

When the interaction among the players is through the state only, a rigorous connection between the $N$-player and the mean field game was first established by \citet{lacker2016general} and \citet{Fischer17}, proving compactness results for sequences of Nash equilibria, using relaxed controls.
\citet{carda15} used existence and regularity properties of solutions of the so-called master equation, a PDE on the space of probability measures characterizing the mean field game to prove convergence results for the value functions.
Moreover, under proper smoothness conditions on the solution (when it exists) of the master equation, \citet{Del-Lac-Ram_Concent,Del-Lac-Ram-19} further analyzed the mean field game limit for the state processes at equilibrium of the game with common noise.
Notably, \citet{Del-Lac-Ram_Concent} derive the large deviation principle for \emph{closed-loop} equilibria.

The mean field game limit in the present context, where interaction is also through the control, was first proved by \citet{pontryagin}.
In this setting we should also mention results by \citet{Bellman-limit} for games in the weak formulation and \citet{DjeteMFG} using relaxed controls.
In the nutshell, it was proved in \cite{pontryagin} that when the coefficients of the game are sufficiently regular (see the proper statement below) and the $N$-player game admits at least one Nash equilibrium, then this Nash equilibrium converges in the $\mathbb{L}^2$--sense to a mean field equilibrium.
The aim of the present work is to make a finer analysis of this convergence result by deriving large deviation principles both for the Nash equilibrium and the state processes of the players at equilibrium.
We will recall the concept of large deviation principle (LDP) in Section \ref{sec:summary.pontryagin}; but note already that this allows to quantify the probability that the Nash equilibrium does not converge to the mean field equilibrium, a rare event in view of the results in \cite{pontryagin}. 
Our main results (presented in the next section) state that if in addition to the convergence assumptions of \cite{pontryagin} we assume that a certain (master) equation admits a Lipschitz--continuous solution, then the Nash equilibrium satisfies the LDP.
In contrast to \citet{Del-Lac-Ram_Concent}, our model includes interaction of the players through their controls, but also deals with \emph{open-loop} equilibria.
In addition, we propose a totally different proof, one that is based on forward backward stochastic differential equations (FBSDE) and their decoupling fields.
We will elaborate more on the idea of proof in the discussion following the statement of Theorem \ref{thm:LDP-Tlarge}.
Note that this method seems versatile enough to be applied to the LDP of \emph{cooperative games}.

When the players jointly optimize the cost function 
\begin{equation*}
	\frac{1}{N}\sum_{i=1}^N\EE\bigg[ \int_0^T f(t,X^{i, \underline \alpha}_t, \alpha^i_t, L^{N}(\underline X^{\underline \alpha}_t, \underline \alpha_t)) dt + g(X^{i, \underline \alpha}_T, L^{N}(\underline X^{\underline\alpha}_T) )\bigg],
\end{equation*}
we obtain a ``central planner'' problem (or cooperative game) which has been showed to converge to a stochastic control problem of McKean-Vlasov type, see \cite{LackerSICON17,pontryagin,Djete2020}.
The analysis of the LDP in this case follows the same steps as in the non-cooperative game described above.
To avoid repeating the proof, in the case of cooperative games we focus on a linear--quadratic game and again derive an LDP, see Section \ref{sec:ldp-mfc} for details.

In the next section we make precise the assumptions used in \cite{pontryagin} as well as the additional assumptions needed for the LDP and state the main results of the paper, Theorems \ref{thm:LDP} and \ref{thm:LDP-Tlarge}.
The proof of these results is the subject of Section \ref{sec:ldp.nash}.
In the last section of the article we consider the case of cooperative stochastic differential games, and restrict ourselves to the linear-quadratic case.

\section{Setting and main results}
\label{sec:summary.pontryagin}

This section is dedicated to the presentation of the main results of this article. 
That is, the large deviation principle for stochastic differential games.
We will focus here on the case of non-cooperative games.
The cooperative case will be discussed in the last section.
The setting is exactly the same as that of \cite{pontryagin}.
We recall it here for the reader's convenience.
Throughout, we denote the set of probability measures with finite second moment on a Polish space $E$ by $\cP_2(E)$, and we equip it with the second order Wasserstein distance denoted $\mathcal{W}_2(\xi, \xi')$ for $\xi,\xi' \in \cP_2(E)$, and defined as
\begin{equation*}
	\cW_2^2(\xi, \xi') := \inf_{\pi}\iint_{E\times E}d_E(x,y)\pi(dx, dy)
\end{equation*}
where $d_E$ is the distance on $E$ and the infimum is taken over coupling of $(\xi, \xi')$, i.e. probability measures $\pi$ on $E\times E$ with first and second marginals $\xi$ and $\xi'$ respectively \footnote{When $E$ is $\mathbb{R}^e$ for some $e \in \mathbb{N}$, $d_E$ is taken to be the usual Euclidean distance.}.
We also denote by $\cP(E)$ the set of probability measures on $E$ and equip this set with the weak topology. 
We denote by $\partial_ah,\partial_{x}h$ the partial derivatives of a function $h$ in the variable of $a\in\mathbb{R}^m,x\in\mathbb{R}^\xdim$.

Recall that a function $\Gamma:\mathcal{P}_2(\mathbb{R}^e)\to \mathbb{R}$ (with $e\in \mathbb{N}$) is $L$--differentiable if there is a continuous function $\partial_\xi\Gamma:\cP_2(\mathbb{R}^e)\times\mathbb{R}^e\to \mathbb{R}$  satisfying the following two properties:
\begin{itemize}
		\item for every $\xi, \xi'\in \cP_2(\mathbb{R}^e)$ it holds
	\begin{equation*}
		\Gamma(\xi) - \Gamma(\xi') = \int_0^1\int_{\mathbb{R}^e}\partial_\xi\Gamma\big((1-t)\xi +t\xi' \big)(x)(\xi' - \xi)(dx)\,dt
	\end{equation*}
	\item $\partial_\xi\Gamma$ is uniformly of quadratic growth on compacts: That is, for every compact set $K \subseteq \cP_2(\mathbb{R}^e)$, there exists a constant $C>0$ such that $|\partial_\xi\Gamma(\xi)(x)| \le C(1 + |x|^2)$ for all $ \xi \in K$ and $x \in \mathbb{R}^e$.
\end{itemize}
see for instance ~\cite{MR2401600,PLLcollege} or~\cite[Chapter 5]{MR3752669} for further details.
In particular, the derivative is uniquely defined, up to an additive constant.
In the rest of the paper, we will use the notation  $\partial_\xi h, \partial_\mu h$ and $\partial_\nu h$ for the $L$-derivative of a function $h$ in the variable of the probability measure $\xi \in \mathcal{P}_2(\mathbb{R}^{\xdim}\times \mathbb{R}^m)$, $\mu \in \mathcal{P}_2(\mathbb{R}^{\xdim})$ and $\nu \in \mathcal{P}_2(\mathbb{R}^{m})$, respectively.

Throughout the paper, $C$ denotes a generic strictly positive constant.
In the computations, the constant $C$ can change from line to line, but this will not always be mentioned.
\emph{However, $C$ will never depend on $N$.}
Consider the following conditions: 
\begin{enumerate}[label = (\textsc{A1}), leftmargin = 30pt]
	\item The function $b:[0,T]\times \mathbb{R}^\xdim\times \mathbb{R}^m\times \mathcal{P}_2(\mathbb{R}^\xdim\times \mathbb{R}^m) \to \mathbb{R}^\xdim$ is continuously differentiable in its last three arguments and satisfies the Lipschitz--continuity and linear growth conditions
	\begin{equation*}
		\begin{split}
		& |b(t,x,a,\xi) - b(t,x', a', \xi')| \le L_b\big(|x-x'| +|a -a'| + \mathcal{W}_2(\xi,\xi')\big)
		\\
		& |b(t,x,a,\xi)| \le L_b\bigg(1+|x| + |a|+ \Big(\int_{\mathbb{R}^{\xdim+m}}|v|^2\,\xi(dv)\Big)^{1/2} \bigg)
		 \end{split}
	\end{equation*}
	for some $L_b>0$ and all $x, x' \in \mathbb{R}^\xdim$, $a,a' \in \mathbb{R}^m$, $t \in [0,T]$ and $\xi,\xi' \in \cP_2(\mathbb{R}^\xdim\times \mathbb{R}^m)$.

	The functions $f:[0,T]\times \mathbb{R}^\xdim\times \mathbb{R}^m\times \mathcal{P}_2(\mathbb{R}^\xdim\times \mathbb{R}^m) \to \mathbb{R}$ and $g:\mathbb{R}^\xdim\times \mathcal{P}_2(\mathbb{R}^\xdim) \to \mathbb{R}$ are of quadratic growth:
	\begin{align*}
		&|f(t,x,a,\xi)| \le L_f\Big(1+|x|^2 + |a|^2 + \int_{\mathbb{R}^{\xdim+m}}|v|^2\,\xi(dv) \Big)\\
		&|g(x,\mu)| \le L_g\Big(1 +|x|^2 + \int_{\mathbb{R}^{\xdim}}|v|^2\,\mu(dv) \Big)
	\end{align*}
	for some $L_f,L_g>0$ and all $x \in \mathbb{R}^\xdim$, $a\in \mathbb{R}^m$, $t \in [0,T]$, $\xi \in \cP_2(\mathbb{R}^\xdim\times \mathbb{R}^m)$ and $\mu \in \cP_2(\mathbb{R}^\ell)$. 
	Moreover, $f$ is continuously differentiable in its last three arguments and $g$ is continuously differentiable.
\label{a1}
\end{enumerate}
\begin{enumerate}[label = (\textsc{A2}), leftmargin = 30pt]
	\item The functions $b$ and $f$ can be decomposed as
	\begin{equation}
	\label{eq:decom bf}
		b(t,x,a,\xi) := b_1(t,x, a, \mu) + b_2(t, x, \xi)\quad \text{and}\quad f(t,x,a,\xi) = f_1(t,x,a,\mu) + f_2(t,x,\xi)
	\end{equation}
	for some functions $b_1,b_2$, $f_1$ and $f_2$, where $\mu$ is the first marginal of $\xi$.
	\label{a2}
\end{enumerate}
\begin{enumerate}[label = (\textsc{A3}), leftmargin = 30pt]
	\item 
	Considering the Hamiltonian
	\begin{equation}
	\label{eq:def-H-hyp}
		H(t,x,y,a,\xi)= f(t,x,a,\xi) + b(t,x,a,\mu)y.
	\end{equation}
	 There is a constant $\gamma>0$ such that
	\begin{equation}
	\label{eq:strong convex}
		H(t,x,y,a,\xi) - H(t,x,y,a',\xi) - (a - a')\partial_aH(t,x,y,a,\xi) \ge\gamma |a-a'|^2
	\end{equation}
	and the functions $x\mapsto g(x, \mu)$ and $(x,a)\mapsto H(t, x, y,a, \xi)$ are convex for all $a,a' \in \mathbb{A}$, $(t,x,y, \xi) \in [0,T]\times \mathbb{R}^\ell\times \mathbb{R}^\ell \times \cP_2(\mathbb{R}^\ell\times \mathbb{R}^m)$.
	In addition, the functions 
	\begin{equation*}
		\text{$\partial_a H(t,\cdot, \cdot,\cdot,\cdot)$, $\partial_xH(t,\cdot, \cdot,\cdot,\cdot)$ and $\partial_xg(\cdot,\cdot)$ are Lipschitz-continuous with Lipschitz constant $L_f$}
	 \end{equation*}
	  and of the linear growth:
	\begin{equation*}
		\begin{cases}
			|\partial_xH(t, x,a,y,\xi)| \le L_f\Big(1 + |x|+ |y| + \big(\int_{\mathbb{R}^{\xdim}}|v|^2\mu(dv)\big)^{1/2} \Big)\\
			|\partial_aH(t, x,a,y,\xi)| \le L_f\Big(1 + |x|+ |a| + |y| + \big(\int_{\mathbb{R}^{\xdim+m}}|v|^2\xi(dv)\big)^{1/2} \Big)\\
			|\partial_xg(x,\mu)| \le L_f\Big(1 +|x| + (\int_{\mathbb{R}^{\xdim}}|v|^2\mu(dv))^{1/2} \Big)
		\end{cases}
	\end{equation*} 
	for all $(t,x,a,\xi)\in [0,T]\times \mathbb{R}^\xdim\times\mathbb{A}\times \mathcal{P}_2(\mathbb{R}^{\xdim}\times \mathbb{R}^m)$, where $\mu$ is the first marginal of $\xi$.
	\label{a3}
\end{enumerate}
\begin{enumerate}[label = (\textsc{A4}), leftmargin = 30pt]
	\item For every $(t,x,a,\xi)\in [0,T]\times \mathbb{R}^\xdim\times\mathbb{A}\times \mathcal{P}_2(\mathbb{R}^{\xdim}\times \mathbb{R}^m)$ and $(u,v) \in \mathbb{R}^\xdim\times\mathbb{R}^m$ we have 
	\begin{equation*}
	\begin{cases}
		| \partial_\mu b(t,x,a,\xi)(u)| \le L_b \\
		 | \partial_\xi f(t,x,a,\xi)(u,v) |\le L_f\Big(1 + |u| + |x| + \Big(\int_{\mathbb{R}^{\xdim}}|v|^2\mu(dv) \Big)^{1/2} \Big)\\
		| \partial_\mu g(x,\mu)(u)| \le L_f\Big(1 + |u|  + |x| +\Big(\int_{\mathbb{R}^\xdim}|v|^2\mu(dv) \Big)^{1/2} \Big)
		\end{cases}
	\end{equation*}
	where $\mu$ is the first marginal of $\xi$.
\label{a4}
\end{enumerate}	
\begin{enumerate}[label = (\textsc{A5}), leftmargin = 30pt]
	\item The matrix $\sigma$ is uniformly elliptic. 
	That is, there is a constant $c>0$ such that
	\begin{equation*}
	\langle \sigma\sigma'x,x\rangle \ge c|x|^2 \quad \text{for every $x \in \mathbb{R}^{\xdim}$.}
\end{equation*}
	\label{a5}
\end{enumerate}
The conditions \ref{a1}-\ref{a5} are essentially regularity and structural conditions on the coefficients of the game.
In \cite{pontryagin}, these conditions are imposed to guarantee convergence of the Nash equilibrium of the $N$-player game to the mean field equilibrium of the associated mean field game. In particular, \ref{a2} is needed due to our method which is based on deriving representations of equilibria in terms of the state process and some adjoint processes.
Such decompositions are typically used in the literature, see e.g. \citet{Car-Del15}. The following is \cite[Theorem 1]{pontryagin}.

\begin{theorem}
\label{thm:main limit}
	Let conditions \ref{a1}-\ref{a5} be satisfied.
	Assume that the $N$-player game admits a Nash equilibrium $\hat{\underline\alpha}^{N} \in \mathcal{A}^N$.
	Then there is $\delta>0$ such that if $T\le \delta$, for each $i = 1,\dots, N$ the sequence $(\hat\alpha^{i,N})$ converges to an admissible control $\hat\alpha^i$ which is a mean field equilibrium and it holds that
	\begin{equation*}
		\mathbb{E}[|\hat{\alpha}^{i,N}_t - \hat\alpha_t^i|^2] \le C r_{N,m,\xdim} 
	\end{equation*}
	for all $N \in \mathbb{N}$ large enough and some constant $C>0$	where, $r_{N,m,\xdim}$ is a rate depending on $N,m,\xdim$ such that $r_{N,m,\xdim}\downarrow 0$ as $N\to \infty$. 
\end{theorem}
The case of arbitrarily large time is treated in that paper under additional monotonicity conditions as in Theorem \ref{thm:LDP-Tlarge} below.
In order to derive a large deviation principle, we will strengthen the growth conditions on the derivatives into boundedness conditions, and more importantly, require the master equation to admit a Lipschitz--continuous classical solution.
\begin{enumerate}[label = (\textsc{A6}), leftmargin = 30pt]
	\item The functions $\partial_xg$, $\partial_\mu g$, $\partial_xf$ and $\partial_\xi f$ are bounded.
	\label{a65}
\end{enumerate}
\begin{enumerate}[label = (\textsc{A7}), leftmargin = 30pt]
	\item There is a measurable function such that
	\begin{equation*}
	 	\Lambda(t, x, y, \mu) \in \argmin_{a \in \mathbb{A}}H(t, x, y,\xi, a),
	\end{equation*}
	where $\mu$ is the first marginal of $\xi$.
	Consider the function $ \varphi_t :\cP_2(\RR^\xdim\times \RR^m) \to \cP_2(\RR^\xdim\times \RR^m)$ given by
	\begin{equation}
	\label{eq:def.phi}
		\varphi_t (\xi) := \xi\circ (id_\xdim, {\Lambda(t,\cdot, \cdot, \mu)})^{-1}
	\end{equation}
	where $id_\xdim$ is the projection on $\RR^\xdim$ and $\mu$ the first marginal of $\xi$ and the functions
	\begin{equation*}
		 B(t, x, y, \xi ) := b\big(t,x, {\Lambda(t,x, y, \mu)}, \varphi_t(\xi)\big),\quad G(x, \mu) = \partial_xg(x,\mu).
	\end{equation*}
	and
	\begin{align*}
		F(t, x, y, \xi) &:=  \partial_{x} f\big(t,x, {\Lambda(t,x, y, \mu)}, \varphi_t(\xi) \big)  +  \partial_{x} b\big(t,x,\Lambda(t,x, y, \mu), \varphi_t(\xi)\big)y
	\end{align*}
	where $\mu$ is the first marginal of $\xi$.
	The following PDE\footnote{Note that the term $\text{tr}(\partial_{xx}V(t,x,\mu)\sigma\sigma')$ is to be understood coordinate-wise, that is, putting $V = (V^1,\dots,V^\xdim)^\top$ we write $\text{tr}(\partial_{xx}V(t,x,\mu)\sigma\sigma') = \left(\text{tr}(\partial_{xx}V^i(t,x,\mu)\sigma\sigma') \right)_{i}$.} admits a (classical) solution $V:[0,T]\times \mathbb{R}^\xdim\times \cP_2(\mathbb{R}^\xdim)\to \mathbb{R}^\ell$ which is Lipschitz--continuous in its second and third arguments uniformly in $t$: 	
	\begin{align}
	\label{eq:master pde}
		\begin{cases}
			\partial_tV(t,x,\mu) + B(t,x,V(t,x,\mu),\xi)\partial_xV(t,x,\mu) + \frac12\text{tr}(\partial_{xx}V(t,x,\mu)\sigma\sigma')
			\\
			\quad + F(t,x,V(t,x,\mu), \xi)
			 \displaystyle + \int_{\RR^d}\partial_\mu V(t,x,\mu)(y)\cdot B(t,y,V(t,x,\mu),\xi)\mu(dy)\\
			 \quad + \int_{\RR^d}\frac12\text{tr}\left(\partial_y\partial_{\mu}V(t,x,\mu)(y) \sigma\sigma'\right)\,\mu(dy) = 0,\qquad (t,x,\mu) \in [0,T)\times \RR^\xdim \times \mathcal{ P}_2(\RR^\xdim)
	 		\\
	 		V(T, x, \mu) = G(x,\mu), \qquad\qquad\qquad\qquad\qquad\qquad\,\,\, (x,\mu) \in \RR^\xdim \times \mathcal{ P}_2(\RR^\xdim)
		\end{cases}
	\end{align}
	where $\xi$ is the joint law of $(\chi, V(t, \chi, \mu))$ when $\cL(\chi) = \mu$.
	\label{a6}
\end{enumerate}

The main result of the present work is a refinement of Theorem \ref{thm:main limit} into a LDP both for a Nash equilibrium and the state processes at equilibrium.
It is well-known by the celebrated Varadhan-Bryc equivalence, given in \cite{bryc90,Dembo-Zeitouni} (see also \cite[Section 1.2]{Dupuis-Ellis97}) that the LDP is equivalent to the so-called Laplace principle which can be stated as follows:
Given a function $\mathcal{I}:\cP(E)\to [0,\infty]$ with (weakly) compact\footnote{Also see \cite[Theorem 4.1]{CommanTAMS} for cases under which the compactness condition can be removed.} sublevel sets $\{\mu \in \cP(E): \mathcal{I}(\mu) \le a \}$ called a (good) rate function, a sequence of measures $(\mu^N)_{N \in \NN}$ on the Polish space $E$ satisfies the Laplace principle (in the weak topology) if for every bounded continuous function $F:\cP(E) \to \RR$ it holds 
\begin{equation*}
	\lim_{N\to \infty}-\frac1N\log(\mathbb{E}\left[\exp(-NF(\mu^N)) \right]) = \inf_{\mu \in \cP(E)}\left(F(\mu) + \mathcal{I}(\mu) \right).
\end{equation*}
In the statement of the result we use the following notation:
\begin{itemize}
	\item The set $\mathcal{U}$ is defined as the set of $((\Omega, \mathcal{F},P), (\mathcal{F}_t)_{t\in [0,T]}, u, W)$ such that the pair $((\Omega, \mathcal{F}, P), (\mathcal{F}_t)_{t\in [0,T]})$ forms a stochastic basis satisfying the usual conditions and carrying the $d$-dimensional Brownian motion $W$ and $u$ is an $\mathbb{R}^d$-valued $(\mathcal{F}_t)_{t\in [0,T]}$-progressive process satisfying $\mathbb{E}\left[\int_0^T|u_t|^2dt\right]<\infty$.
	\item By $\cC^e$ we denote the space of continuous maps from $[0,T]$ to $\mathbb{R}^e$.
	\item The map $\overline B:[0,T]\times \RR^\xdim \times \cP_2(\RR^\xdim)\to \RR^\ell$ is defined as
	\begin{equation}
	\label{eq:def.barB}
		\overline B(t,x,\mu) := B\Big(t, x, V(t, x, \mu), \cL\big(\chi, V(t, \chi, \mu) \big) \Big)
	\end{equation} where $\chi\in \mathbb{L}^2(\Omega, \cF,\PP)$ is an $\mathbb{R}^\xdim$-valued random variable with law $\cL(\chi) = \mu$, with $B$ defined in \ref{a6}.
	\item The map $\Psi:[0,T]\times \cP_2(\RR^\xdim)\to \cP_2(\RR^m)$ is defined as
	\begin{equation}
	\label{eq:def.Psi}
		\Psi( t, \mu) := \mu \circ \Lambda\big(t, \cdot, V(t, \cdot, \mu),\mu \big)^{-1} .
	\end{equation}
\end{itemize}
\begin{theorem}
\label{thm:LDP}
	If the conditions \ref{a1}-\ref{a6} are satisfied and the $N$-player game admits a Nash equilibrium $\underline{\hat\alpha}$, then, there is a constant $c(L_b,L_f,L_g)$ depending on (the Lipschitz constants of) $f,$ $b$ and $g$ such that if $T\le c(L_b,L_f,L_g)$, then the following hold:
	\begin{enumerate}
		\item[(i)] The sequence $(L^N(\underline X^{\underline{\hat\alpha}}))_N$ satisfies the LDP on $\cP(\cC^\xdim)$ with rate function\footnote{We use the convention $\inf\emptyset = \infty.$}
	\begin{equation}
	\label{eq:rate.X}
		\mathcal{I}(\theta) = \inf_{u\in \mathcal{U}: \mathrm{law}(X^u)=\theta}\EE\bigg[\frac12\int_0^T|u_t|^2\,dt\bigg], \quad \theta \in \cP(\mathcal{C}^\ell)
	\end{equation}
	where $dX^u_t = \overline B(t, X^u_t, \cL(X^u_t)) + \sigma u_t\,dt + \sigma\,dW_t$.
		\item[(ii)] If in addition, the functions $\Lambda:[0,T]\times \mathbb{R}^\xdim\times\mathbb{R}^\xdim \times \cP(\mathbb{R}^\xdim)\to \mathbb{R}^m $ and $V:[0,T]\times \mathbb{R}^\xdim\times \cP(\mathbb{R}^\xdim)\to \RR^\xdim$ are continuous on $\mathbb{R}^\ell\times \cP(\mathbb{R}^\xdim)$ for all $t\in [0,T]$, then the sequence $(L^N(\underline{\hat \alpha}_t))_N$ satisfies the LDP on $\cP(\RR^m)$ with rate function
		\begin{equation*}
			\mathcal{\widetilde I}_t(\nu) :=\inf_{\theta \in \cP_2(\mathcal{C}^\ell): \Psi(t,\theta_t) = \nu}\mathcal{I}(\theta),\quad \nu \in \cP(\mathbb{R}^m),
	\end{equation*} 
	where $\theta_t$ is the time $t$ marginal of $\theta$.
	\end{enumerate}
\end{theorem}

In summary, Theorem \ref{thm:LDP} tells us that when the coefficients $b,f$ and $g$ of the game are sufficiently regular, the Hamiltonian satisfies a certain convexity condition and the master equation corresponding to the mean field game admits a Lipschitz--continuous solution, then, not only that any sequence of Nash equilibria converges to a mean field equilibrium, but in addition the sequence of Nash equilibria satisfies the LDP.
The condition pertaining to existence and Lipschitz--continuity of the solution of the master equation \eqref{eq:master pde} (i.e. \ref{a6}) is the only hard-to-check condition we impose here.
That being said, conditions guarantying existence and uniqueness of the master equation are given by \citet{ChassagneuxCrisanDelarue_Master} and \citet{carda15}.
These authors study for instance the equation arising from a mean field game (albeit without interaction through the controls) in \cite[Section 5]{ChassagneuxCrisanDelarue_Master}.
In section \ref{sec:example} below we discuss an example for which the equation is known (from the work \cite{ChassagneuxCrisanDelarue_Master}) to have a Lipschitz--continuous solution.
Let us also refer to \cite{Gang-Mes-Mou-Zha,Card-Soug20,Mou-Zhan,Bay-Cecc-Coh-Del} for more recent results on the existence of the master equation.
Moreover, let us observe that \cite{Del-Lac-Ram_Concent} make similar assumptions, to ours, and additionally assume the Hamilton--Jaboci--Bellman system characterizing the $N$--player game to have well--behaved (classical) solutions.
Finally, observe that obtained here is very similar to the rate function given in terms of weak solutions of McKean--Vlasov equations first derived in \cite{Bud-Dup-Fish}.

The limitation in the above theorem is to assume $T$ small enough, a condition which is needed to guarantee some FBSDE estimations.
We can get around the smallness condition by imposing additional monotonicity--type conditions on the parameters.
In fact, consider the following condition:
\begin{enumerate}[label = (\textsc{A8}), leftmargin = 30pt]
	\item With the function $\Lambda$ defined in \ref{a6}, the drift $b$ satisfies the monotonicity condition
	\begin{equation}
	\label{eq:mon.con.b}
			(x-x')\cdot\Big(b(t,x,\Lambda(t, x, y, \mu),\xi)-b(t,x',\Lambda(t, x', y, \mu),\xi) \Big)\le -K_b|x-x'|^2
	\end{equation}
	and the functions $b,H$ and $g$ satisfy
	\begin{equation}
	\label{eq:mon.con.b.H.g}
		\begin{cases}
			(y-y')\cdot\Big( b(t,x,\Lambda(t, x, y, \mu),\xi)-b(t,x,\Lambda(t, x, y',\mu),\xi) \Big) \le -K|y-y'|^2\\
			(x - x')\cdot\Big( \partial_xH(t,x',\Lambda(t, x', y, \mu),\xi)-\partial_xH(t,x,\Lambda(t, x, y,\mu),\xi) \Big) \le -K|x-x'|^2\\
			(x -x')\cdot\Big(\partial_xg(x, \mu) - \partial_xg(x',\mu)\Big)\ge K|x-x'|
		\end{cases}
	\end{equation}
	for all $t \in [0,T]$, $x,x', y, y' \in \RR^\xdim$, $a \in \mathbb{A}$ and $\xi \in \cP_2(\RR^{\xdim\times m})$, and for some constants $K,K_b>0$.
	\label{a8}
\end{enumerate}
Assuming monotonicity of $\Lambda$ is not an abstract condition.
In many cases (e.g.\ in the linear quadratic case) $\Lambda$ is a linear function of $y$.
Note moreover that by \eqref{eq:strong convex} it is easily checked that the function $\Lambda$ is Lipschitz--continuous (see e.g. \cite{pontryagin}).
We denote by $L_\Lambda$ the Lipschitz constant of $\Lambda$.
 We will further distinguish the Lipschitz constant of $b$ in each of its arguments. 
 Thus, we denote by $L_{b,x},L_{b,a},L_{b,\xi}$ the Lipschitz constant of $b$ in the variables $x, a,\xi$, respectively.
Under the above additional assumption, we have the following LDP:
\begin{theorem}
\label{thm:LDP-Tlarge}
	If the conditions \ref{a1}-\ref{a8} are satisfied and the $N$-player game admits a Nash equilibrium $\underline{\hat\alpha}$, then, for arbitrarily large $T>0$, there is a constant $c(T,L_{b,a},L_{b,\xi},L_\Lambda)>0$ depending only on $T,L_{b,a},L_{b,\xi}$ and $L_\Lambda$ such that if $K_b> c(T,L_{b,a},L_{b,\xi},L_\Lambda)$ then the conclusions $(i)$ and $(ii)$ of Theorem \ref{thm:LDP} hold.
\end{theorem}

The idea of the proofs is inspired from \cite{Del-Lac-Ram_Concent}, but the details and techniques are wholly different.
The proof starts by identifying a weakly interacting particle system for which the LDP is known, and that is exponentially close (see Definition \ref{def:expo.close} below) to the state processes, then use the closeness property to "transfer" the LDP to the state processes.
Since the state processes of the $N$ agents at equilibrium are characterized by a system of forward-backward SDEs, the identification of the suitable auxiliary particle system whose LDP is known uses the well-known technique of \emph{decoupling fields}.
However, the proper decoupling field turns out to be the solution (when it exists) of the \emph{master equation}.
The main difficulty lies in the proof of the exponential closeness property.
This is based on a priori bounds for systems of FBSDEs at least provided that the solution of the master equation is Lipschitz--continuous.
In addition to the fact that we consider games allowing control interaction and open-loop controls,
the essential difference with \cite{Del-Lac-Ram_Concent} is the assumptions made on the PDEs.
On the one hand, we do not make use of the $N$-player PDE, and on the other hand we require the solution of the master equation to be \emph{Lipschitz--continuous}; we do not make additional regularity assumptions pertaining to its second derivative.
This is an interesting by-product of the convergence method of \cite{pontryagin} which is based on Pontryagin's stochastic maximum principle.
The (very) rough intuition for this gain of smoothness is that in the present case, the FBSDE (resp. the master equation) allows to represent the control (or the "derivative of the value function").
In contrast, the Hamilton-Jacobi-Bellman equation used in \cite{Del-Lac-Ram_Concent} represents the value function itself. 

\section{ Laplace principle for non-cooperative games}
\label{sec:ldp.nash}

This section is mostly dedicated to the proof of Theorem \ref{thm:LDP}.
We will start by shortly recalling the main idea of the proof of \cite[Theorem 1]{pontryagin}.
This will prepare the terrain for us here for the proof of the main results of the present paper.
The proof starts with key representation results for the Nash equilibrium and the mean field equilibrium.
In fact, let $\hat\alpha^{i,N}$ be a Nash equilibrium.
It follows from \cite{pontryagin} that there is a Lipschitz--continuous function $\widehat\Lambda:[0,T]\times \RR^\ell\times \RR^\ell \times \cP_2(\RR^\ell)\times \RR \to \AA$ such that
\begin{equation}
\label{eq:N Nash}
	\hat\alpha_t^{i,N} = \widehat\Lambda(t,X^{i,\underline{\hat\alpha}}_t, Y^{i,i}_t, L^N(\uX_t^{\underline{\hat\alpha}}),\zeta^{i,N}_t),
\end{equation}
and $(Y^{i,j}, Z^{i,j,k})$ satisfies the system of adjoint equations
\begin{equation}
\label{eq:Y.ij}
	\begin{cases}
		dX^{i,\underline{\hat\alpha}}_t = b(t, X^{i, \underline{\hat\alpha}}_t,\hat\alpha_t^{i,N},L^N(\underline{X}^{\underline{\hat\alpha}}_t,\underline{\hat\alpha}_t) )\,dt +\sigma\,dW_t^i\\
		d Y^{i,j}_t = -\partial_{x^j}H^{N,i}(t, X^{i, \underline{\hat\alpha}}_t, \hat\alpha^{i,N}_t, \underline{Y}^{i,\cdot}_t)\,dt + \sum_{k=1}^N Z^{i,j,k}_t dW^{k}_t   \\
		X^{i,\underline{\hat\alpha}}_0=x,\quad		\hat\alpha^{i,N}_t = \widehat\Lambda\Big(t,X^{i,\underline {\hat\alpha}}_t, Y^{i,i}_t,L^{N}(\underline X^{ \underline{\hat\alpha}}_t), \zeta^{i,N}_t \Big),\quad Y^{i,j}_T = \partial_{x^j}g^{N,i}(\underline X^{i, \underline{\hat\alpha}}_T)
	\end{cases}
\end{equation}
with 
\begin{equation*}
	H^{N,i}(t, \underline{x}, \underline{\alpha}, \underline{y}) : = f(t, x^i, \alpha^i, L^N(\underline x, \underline\alpha)) + \sum_{j=1}^Nb(t, x^j, \alpha^j,L^N(\underline x, \underline \alpha))y^{i,j}\text{ and } g^{N,i}(\underline x) = g(x^i, L^N(\underline x))
\end{equation*}
and thus\footnote{As usual $\delta_{\{i=j\}}= 1$ if $i=j$ and  $\delta_{\{i=j\}}= 0$ if $i\neq j$.}
\begin{equation*}
	\begin{cases}
		\partial_{x^j}H^{N,i}(t, X^{i, \underline{\hat\alpha}}_t, \hat\alpha^{i,N}_t, \underline{Y}^{i,\cdot}_t) = \delta_{\{i=j\}}\partial_xf(t, X^{i, \underline{\hat\alpha}}_t,\hat\alpha_t^{i,N},L^N(\underline{X}^{\underline{\hat\alpha}}_t,\underline{\hat\alpha}_t) ) + \partial_xb(t, X^{i, \underline{\hat\alpha}}_t,\hat\alpha_t^{i,N},L^N(\underline{X}^{\underline{\hat\alpha}}_t,\underline{\hat\alpha}_t) )Y^{i,j} + \varepsilon^{i,N}\\
		\partial_{x^j}g^{N,i}(\underline{X}_T^{i,\underline{\hat\alpha}}) = \delta_{\{i=j\}}\partial_xg(X^{i,\underline{\hat\alpha}}, L^N(\underline{X}_T^{i,\underline{\hat\alpha}})) + \gamma^{i,N}
	\end{cases}
\end{equation*}
where
\begin{equation*}
	\begin{cases}
		\varepsilon^{i,N}_t := \frac1N\partial_\mu f(t, X^{i, \underline{\hat\alpha}}_t, \hat\alpha^{i,N}_t, L^N(\underline{X}^{\underline{\hat\alpha}}_t,\underline{\hat\alpha}_t ))(X^{i, \underline{\hat\alpha}}_t) + \frac1N\sum_{j=1}^N\partial_\mu b(t, X^{j, \underline{\hat\alpha}}_t, \hat\alpha^{j,N}_t, L^N(\underline{X}^{\underline{\hat\alpha}}_t,\underline{\hat\alpha}_t ))(X^{i, \underline{\hat\alpha}}_t)Y^{i,j}_t\\
		\zeta^{i,N}_t := -\frac1N\partial_\nu f_2(t,X_t^{i,\underline{\hat\alpha}}, L^{N}(\underline X^{\underline{\hat\alpha}}_t,\underline{\hat\alpha}_t) )(\hat\alpha^{i,N}_t) -\frac1N\sum_{k=1}^N\partial_\nu b_2(t,X_t^{i,\underline{\hat\alpha}}, L^{N}(\underline X^{\underline{\hat\alpha}}_t,\underline{\hat\alpha}_t) )(\hat\alpha^{k,N}_t)Y^{i,k}_t\\
		\gamma^{i,N} := \frac1N\partial_\mu g(X^{i, \underline{\hat\alpha}}_T, L^N(\underline{X}^{\underline{\hat\alpha}}_T))(X^{i, \underline{\hat\alpha}}_T).
	\end{cases}
\end{equation*}
We will particularly be interested in the diagonal term $(X^{i,\underline{\hat\alpha}},Y^{i,i}, Z^{i,i,k})$ which, by a quick verification, can be shown to satisfy
\begin{equation}
\label{eq:true.system.N}
	\begin{cases}
		dX^{i,\underline{\hat\alpha}}_t = b(t, X^{i, \underline{\hat\alpha}}_t,{\hat\alpha_t^{i,N}},L^N(\underline{X}^{\underline{\hat\alpha}}_t,\underline{\hat\alpha}_t) )\,dt +\sigma\,dW_t^i\\
		d Y^{i,i}_t = -\Big\{\partial_xf(t, X^{i, \underline{\hat\alpha}}_t, \hat\alpha^{i,N}_t, L^N(\underline{X}^{\underline{\hat\alpha}}_t,\underline{\hat\alpha}_t )) + \partial_xb(t, X^{i, \underline{\hat\alpha}}_t, \hat\alpha^{i,N}_t, L^N(\underline{X}^{\underline{\hat\alpha}}_t,\underline{\hat\alpha}_t ))Y^{i,i}_t+ \varepsilon^{i,N}_t \Big\}\,dt + \sum_{k=1}^N Z^{i,i,k}_t dW^{k}_t   \\
		X^{i,\underline{\hat\alpha}}_0=x,\quad		\hat\alpha^{i,N}_t = \widehat\Lambda\Big(t,X^{i,\underline {\hat\alpha}}_t, Y^{i,i}_t,L^{N}(\underline X^{ \underline{\hat\alpha}}_t), \zeta^{i,N}_t \Big),\quad Y^{i,i}_T = \partial_xg(X^{i, \underline{\hat\alpha}}_T, L^N(\underline{X}^{\underline{\hat\alpha}}_T)) +\gamma^{i,N}.
	\end{cases}
\end{equation}
On the other hand, the limiting mean field equilibrium $\hat\alpha$ satisfies
\begin{equation}
\label{eq:mfe.rep}
	\hat\alpha^i_t  =  \widehat\Lambda(t,X^i_t, Y^i_t, \cL(X^i_t),0) \equiv \Lambda(t,X^i_t, Y^i_t, \cL(X^i_t))
\end{equation}
where\footnote{For ease of notation we omit the superscript $i$ and write $(X, Y, Z, W)$ instead of $(X^i, Y^i, Z^i,W^i)$; and $\hat\alpha$ instead of $\hat\alpha^i$.} $(X,Y,Z)$ solves the Mckean--Vlasov FBSDE
\begin{equation}
\label{eq:true.McKV.eq}
	\begin{cases}
		dX_t = b( X_t,\hat\alpha_t,\cL(X_t,\hat\alpha_t) )\,dt +\sigma\,dW_t\\
		d Y_t = -\Big\{\partial_xf(X_t, \hat\alpha_t, \cL(X_t,\hat\alpha_t )) + \partial_xb( X_t, \hat\alpha_t, \cL(X_t,\hat\alpha_t ))Y_t \Big\}\,dt +  Z_t dW_t   \\
		X_0=x,\quad		\hat\alpha_t=\Lambda\Big(t,X_t, Y_t,\cL( X_t) \Big),\quad Y_T = \partial_xg(X_T, \cL(X_T)).
	\end{cases}
\end{equation}
More precisely, it follows by propagation of chaos arguments that $(X^{i,\hat\alpha}, Y^{i,i})$ converges to $(X,Y)$ in $\mathcal{S}^2(\RR^\xdim\times \RR^\xdim)$, where, given a normed space $E$ and $p\ge0$, we denote by $\mathcal{S}^p(E)$ the space of adapted processes $X$ equipped with the norm
\begin{equation*}
	\|X\|_{\mathcal{S}^2(E)} := \EE\Big[\sup_{t \in [0,T]}\|X_t\|^p_E\Big].
\end{equation*}
With this preparation out of the way, we are now ready for the proofs of the main results of this article.
The next section focuses on the non-cooperative $N$-player game described in the introduction, and Section \ref{sec:ldp-mfc} will deal with the linear-quadratic case for cooperative games.

\subsection{LDP for small time horizons: Proof of Theorem \ref{thm:LDP}}

As announced in the previous section, the proof of Theorem \ref{thm:LDP} builds upon LDP for uncontrolled, interacting (forward) particle systems.
In order to exploit such results, we will introduce a forward-backward particle system that is "similar" to (but considerably more tractable than) \eqref{eq:true.system.N}.
Using this auxiliary particle system and well-known decoupling techniques from the theory of forward-backward SDEs allow to construct an uncontrolled forward particle system for which LDP results are well-known.
The last step of the proof is to show that the auxiliary particle system for which the LDP is known in the literature is "close enough" to our original particle system.
This is achieved by using a priori estimations for FBSDEs.
Here, close enough should be understood in the following sense put forth in \cite[Definition 4.2.10]{Dembo-Zeitouni}:
\begin{definition}
\label{def:expo.close}
	Let $(\mathcal{Y}, d)$ be a metric space.
	The probability measures $\mu_\varepsilon$ and $\tilde \mu_\varepsilon$ on $\mathcal Y$ are called exponentially equivalent if there exist probability spaces $(\Omega, \mathcal B_\varepsilon, P_\varepsilon)$ and two families of $\mathcal Y$-valued random variables $Z_\varepsilon$ and $\tilde Z_\varepsilon$ with joint laws $(P_\varepsilon)$ and marginals $\mu_\varepsilon$ and $\tilde \mu_\varepsilon$, respectively, such that the following condition is satisfied:

	For each $\delta>0$, the set $\{\omega: (\tilde Z_\varepsilon, Z_\varepsilon) \in \Gamma_\delta\}$ is $\mathcal B_\varepsilon$ measurable \footnote{The measurability requirement is satisfied whenever $\mathcal{Y}$ is a separable space, or whenever the laws $\{P_\varepsilon\}$ are induced by separable real-valued stochastic processes and $d$ is the supremum norm, see Remarks below \cite[Definition 4.2.10]{Dembo-Zeitouni}.}, and
	\begin{equation*}
		\limsup_{\varepsilon\to 0}\varepsilon\log P_\varepsilon(\Gamma_\delta) = - \infty,
	\end{equation*}
	where $\Gamma_\delta = \{(\tilde y, y): d(\tilde y, y) >\delta\} \subseteq \mathcal Y\times \mathcal Y$.
 \end{definition}

\begin{proof}[Proof of Theorem \ref{thm:LDP}]
	At equilibrium, the state process of player $i$ is given by the SDE
	\begin{equation*}
		dX^{i,\underline{\hat\alpha}}_t = b(t,X^{i,\underline{\hat\alpha}}_t, \hat\alpha^{i,N}_t, L^N(\uX_t^{\underline{\hat\alpha}}, \hat{\ua}_t))\,dt + \sigma\,dW^i_t
	\end{equation*}
	with $\hat\alpha^{i,N}$ given by \eqref{eq:N Nash}.
	Since $\Lambda$ is Lipschitz--continuous (this follows from \eqref{eq:strong convex}), it can be shown that $\varphi_t$ defined in \eqref{eq:def.phi} is Lipschitz continuous with respect to the second order Wasserstein distance, see for instance (the proof of) \cite[Theorem 1]{pontryagin}.
	Thus, the functions $B,F$ and $G$ introduced in \ref{a6} are Lipschitz--continuous and of linear growth.
	Further observe that using these functions, the McKean--Vlasov equation \eqref{eq:true.McKV.eq} characterizing the mean field equilibrium reads
	\begin{equation}
	\label{eq:MckV fbsde}
	 	\begin{cases}
	 		d X_t = B(t,  X_t, Y_t, \cL(X_t, Y_t))\,dt + \sigma\,dW_t\\
	 		dY_t = - F(t, X_t, {Y}_t, \cL( X_t, Y_t))\,dt +  Z_t\,dW_t\\
	 		X_0 =x, \quad  Y_T = G(X_T, \cL(X_T)).
	 	\end{cases}
	 \end{equation}

	We will now introduce two auxiliary interacting particle systems that will allow us to derive the LDP for the sequence of interest.
	First consider the equation
	\begin{equation}
	\label{eq:aux.fbsde}
	 	\begin{cases}
	 		d\widetilde X^{i,N}_t = B(t,\widetilde X^{i,N}_t, \widetilde Y^{i,N}_t, L^N(\widetilde {\underline{X}}_t, \widetilde{\underline Y}_t))\,dt + \sigma\,dW_t^i\\
	 		d\widetilde Y^{i,N}_t = - F(t,\widetilde X^{i,N}_t, \widetilde {Y}^{i,N}_t, L^N(\widetilde {\underline X}_t, \widetilde{\underline Y}_t))\,dt + \sum_{k=1}^N\widetilde Z^{i,k}_t\,dW_t^k\\
	 		X_0 = x, \quad  \widetilde Y_T^{i,N} = G(\widetilde{X}^{i,N}_T,L^N(\underline {\widetilde X}_T))
	 	\end{cases}
	 \end{equation}
	which simply corresponds to \eqref{eq:true.system.N} after taking "$\varepsilon^{i,N} = \gamma^{i,N}= \zeta^{i,N}=0$" for all $i$.
	Note that by Lipschitz-continuity of $B,F$ and $G$, it follows e.g. from \cite[Theorem 4.2]{MR3752669} that if $T$ is small enough, then Equation \eqref{eq:aux.fbsde} admits a unique solution.
		
	The second auxilliary equation is introduced through a decoupling argument.
	Since $V$ is a classical solution of the PDE \eqref{eq:master pde}, applying It\^o's formula to $\bar Y_t := V(t, X_t, \cL(X_t))$ shows that there is $\bar Z$ such that $(\bar Y, \bar Z)$ solves the backward equation in \eqref{eq:MckV fbsde}.
	Thus, by uniqueness, we have $Y_t = V(t, X_t, \cL(X_t))$.
	The function $V$ is often called a decoupling field for the system \eqref{eq:MckV fbsde} because it allows to write the system as two decoupled equations, where $X$ satisfies
	\begin{align}
	\notag
		dX_t &= B\Big(t,X_t, V(t, X_t, \cL(X_t)), \cL\big(X_t, V(t, X_t, \cL(X_t)))\big) \Big)\,dt + \sigma\,dW_t\\
	\label{eq:mkv sde}
		 &= \overline B(t,X_t, \cL(X_t))\,dt + \sigma\,dW_t 
	\end{align}
	where $\overline B$ is the function defined in \eqref{eq:def.barB}.
	Since $V(t, \cdot,\cdot)$ is Lipschitz--continuous on $\RR^\xdim \times \cP_2(\RR^\xdim)$ uniformly in $t\in [0,T]$, it follows that  $\overline B$ is Lipschitz continuous on $\RR^\xdim\times\cP_2(\RR^\xdim)$ as well.
	The second auxiliary particle system is then
	\begin{equation}
	\label{eq:SDE standard}
		d\cX^{i,N}_t = \overline B(t,\cX^{i,N}_t, L^N(\underline\cX_t))\,dt + \sigma\,dW^i_t, \quad  \cX^{i,N}_0 =x.
	\end{equation}
	Note that $(\cX^{i,N})_{i=1,\dots,N}$ is well-defined by classical SDE theory.
	Moreover, the standard theory of propagation of chaos (see e.g.\ \cite{MR1108185}) shows that the sequence $\cX^{i,N}$ converges in $\mathcal{S}^2(\RR^\xdim)$ to $X$ and the sequence of empirical measures $L^N(\underline{\cX}_t)$ converges to $\cL(X_t)$ in $\cP_2(\RR^\xdim)$, see e.g. \cite[Theorem 2.12]{MR3753660}.
	Furthermore, again by Lipschitz--continuity of $\overline B$, it follows by \cite[Theorem 5.2]{Fischer14} (see also \cite[Theorem 3.1]{Bud-Dup-Fish}) that the family of empirical measures $(L^N(\underline\cX))_N$ satisfies the Laplace principle (in $\cP(\cC^\xdim)$) with rate function given by \eqref{eq:rate.X}.
	\cite[Theorem 5.2]{Fischer14} shows that this rate function is lower semicontinuous for the weak topology, and therefore the sublevel sets are weakly closed.
	Let us show that the rate function is good, i.e. that it has weakly compact sublevel sets.
	By Prokhorov's theorem, if we show that the sublevel sets are tight it will follow that they are weakly relatively compact.
	Let $\mathbb{Q}^n$ be a probability measures on (the Polish space) $\cC^d$ such that $\mathcal{I}(\mathbb{Q}^n)\le K$ for some constant $K>0$.
	Then there is a sequence $(u^n)$ in $\mathcal{H}^2(\RR^d)$, (the space of $\RR^d$--valued, square integrable and progressive processes) such that
	\begin{equation*}
		\EE\bigg[\frac12\int_0^T|u^n_t|^2\,dt\bigg]< K + 1/n\quad \text{and}\quad \cL(X^{u^n}) = \mathbb{Q}^n.
	\end{equation*}
	Thus, it suffices to show that the sequence $\cL(X^{u^n})$ is tight.
	Observe that by boundedness of the sequence $(u^n)$ in $\mathcal{H}^2(\RR^d)$ and linear growth of $\overline B$ it follows by standard SDE estimations that $\sup_n\sup_t\EE[|X^{u^n}_t|^2]<\infty$.
	Furthermore, for every $0\le s\le t\le T$, we have
	\begin{align*}
		|X^{u^n}_t - X^{u^n}_s| \le C\int_s^t1 + |X^{u^n}_r| + \EE[|X^{u^n}_r|^2]^{1/2} + |u^n_r|\,dr + |\sigma||W_t - W_s|.
	\end{align*}
	Hence, taking expectation on both sides and applying Cauchy-Schwarz inequality yields
	\begin{align*}
		\EE|X^{u^n}_t - X^{u^n}_s| &\le C|t-s|^{1/2}\bigg(1 + \EE\bigg[\int_0^T|u^n_r|^2\,dr\bigg]^{1/2} + \sup_r\EE\big[|X^{u^n}_r|^2\big]^{1/2} + |\sigma|\bigg)\\
		&\le C|t-s|^{1/2}.
	\end{align*}
	Thus, since $(X^{u^n})$ is a sequence of continuous processes, it follows by Kolmogorov's tightness criterion that $\mathbb{Q}^n$ is tight.
	Therefore, $\mathcal{I}$ is a good rate function.

	In order to "transfer" the LDP from the sequence $(L^N(\underline{\cX}))$ to the relevant sequence $(L^N(\underline{X}^{\underline{\hat\alpha}}))$, we need to show that the two sequences are exponentially close in the sense of Definition \ref{def:expo.close}.
	This follows from Chebyshev's inequality and Lemma \ref{lem:bound_Y} below since we have
	 \begin{align*}
	 	\PP\Big(\sup_{t\in [0,T]}\cW_2(L^N(\uX_t^{\underline{\hat\alpha}}), L^N(\underline\cX_t)) >\varepsilon\Big) &\le \PP\Big( \Big\{\frac1N\sup_{t\in [0,T]}\sum_{i=1}^N|X^{\underline{\alpha},i}_t -  \cX^{i}_t|^2\Big\}^{1/2} \ge \varepsilon \Big)\\
	 	&\le \EE\Big[\exp\Big\{ \Big(\sum_{i=1}^N\|X^{\underline{\alpha},i} -  \cX^{i}\|^2_\infty \Big)\Big\} \Big]e^{-\varepsilon^2N^2}\\
	 	&\le C e^{-\varepsilon^2 N^2}. 
	\end{align*}

	Therefore, 	 	
	\begin{equation*}
	 	\lim_{N\to \infty}\frac1N\log \PP\Big(\sup_{t\in [0,T]}\cW_2(L^N(\uX_t^{\underline{\hat\alpha}}), L^N(\underline\cX_t)) >\varepsilon \Big) = - \infty.
	 \end{equation*}
	It then follows from \cite[Theorem 4.2.13]{Dembo-Zeitouni} that the sequence $L^N(\uX^{\underline{\hat\alpha}})$ satisfies the LDP with rate function $\mathcal{I}$.

\vspace{.2cm}

	Let us now turn to the large deviation principle for the $N$-Nash equilibrium $(\hat\alpha^{i,N})_{i=1,\dots,N}$.
	The difficulty here is the fact that $\hat\alpha^{i,N}$ is not a function of $X^{i,\underline{\hat\alpha}}$ and $L^N(\underline{X}^{\underline{\hat\alpha}})$ only, it also depends on the process $\zeta_t^{i,N}$, see \eqref{eq:true.system.N}.
	Nevertheless, using the contraction principle, we will again prove the LDP for an auxiliary sequence that is exponentially equivalent to $L^N(\underline{\hat\alpha})$.
	Define the auxiliary process
	\begin{equation*}
		\alpha^{i,N}_t := \Lambda\Big(t,\cX^{i,N}_t, V\big(t,\cX^{i,N}_t, L^N(\underline{\cX}_t)\big),L^N(\underline{ \cX}_t) \Big).
	\end{equation*}
	It follows by continuity of $\Lambda$ and $V$, the convergence of $(\cX^{i,N})_{N\ge1}$ to $X$ and the representation \eqref{eq:mfe.rep} of the mean field equilibrium $\hat\alpha$ that\footnote{Recall that to simplify notation we write $\hat\alpha$ instead of $\hat\alpha^i$.}
	\begin{equation*}
		 \alpha^{i,N}_{t} \to \hat\alpha_t \quad \text{in}\quad \mathbb{L}^2.
	\end{equation*}
	We will now use the contraction principle to show that $(L^N(\underline{\alpha}_t))_{N\ge1}$ satisfies the LDP.
	Consider the function $\Psi$ mapping $\cP(\RR^\xdim)$ to $\cP(\RR^\xdim)$ and defined as
	\begin{equation}
		\Psi( t,\mu) := \mu \circ \Lambda\big(t, \cdot, V(t, \cdot, \mu),\mu \big)^{-1} .
	\end{equation}
	The function $\Psi$ is continuous.
	In fact, given a sequence $(\mu^n)_{n\ge1}$ converging to $\mu$ in the weak topology, and a bounded, Lipschitz-continuous test function $f$, we have
	\begin{align*}
		\int_{\RR^\xdim}f(x)\Psi(t,\mu^n)(dx) &= \int_{\RR^\xdim}f\Big(\Lambda(t, x, V(t, x, \mu^n),\mu^n)\Big)\mu^n(dx). 
	\end{align*}
	Since $\Lambda$ and $V$ are continuous on $\cP(\RR^\xdim)$, the functions $f^n(x):=f\Big(\Lambda(t, x, V(t, x, \mu^n),\mu^n)\Big)$ define a sequence of bounded continuous functions converging to $f\Big(\Lambda(t, x, V(t, x, \mu),\mu)\Big)$ pointwise.
	Since $f^n$ is uniformly bounded and uniformly Lipschitz, it follows that
		\begin{align*}
			\int_{\RR^\xdim}f(x)\Psi(t,\mu^n)(dx)&= \int_{\RR^\xdim}f^n(x)\mu^n(dx) \to \int_{\RR^\xdim}f\Big(\Lambda(t, x, V(t, x, \mu),\mu)\Big)\mu(dx),
		\end{align*}
	showing that $\Psi$ is continuous.
	Now, by the definition of $\Psi$, we have
	$$L^N(\underline{\alpha}_t) = \Psi(t, L^N(\underline{\cX}_t))$$
	and
	similarly, since the mean field equilibrium $\hat\alpha$ satisfies the representation $\hat\alpha_t =\Lambda(t,X_t, Y_t, \cL(X_t)) = \Lambda\big(t, X_t, V(t, X_t, \cL(X_t)), \cL(X_t) \big)$, it follows that $$\cL(\alpha_t) = \Psi(t, \cL(X_t)).$$
	Using the fact that the sequence $L^N(\underline {\cX}_t)$ satisfies the LDP with rate function
	\begin{equation*}
	 	\mathcal{I}_t(\mu) := \inf_{\theta \in \cP(\cC^\ell): \theta_t = \mu}\cI(\theta)
	 \end{equation*} and that $\Psi$ is continuous, it follows by the contraction principle that $L^N(\ua_t)$ satisfies the LDP as well.
	In fact, for every bounded continuous function $F:\cP(\mathbb{R}^m) \to \mathbb{R}$, the function $F\circ \Psi(t,\cdot)$ is again bounded continuous and therefore we have
	\begin{align*}
	 	-\frac1N\log \EE\Big[\exp(-NF(L^N(\underline{\alpha}_t)))\Big] &= -\frac1N\log \EE\Big[\exp\big(-NF\circ\Psi(t,L^N(\underline{\cX}_t))\big)\Big]\\
	 		 &\to \inf_{\mu\in \cP(\mathbb{R}^\ell)}(F\circ\Psi(t,\mu) + \mathcal{I}_t(\mu)) \\
	 		 & = \inf_{\nu\in \cP(\mathbb{R}^m)}(F(\nu) + \widetilde {\mathcal{I}}(\nu))
	 \end{align*}
	 with
	\begin{equation*}
		\widetilde {\mathcal{I}}(\nu) :=\inf_{\theta \in \cP(\mathcal{C}^\ell): \Psi(t,\theta_t) = \nu}\mathcal{I}(\mu).
	\end{equation*} 

	Similar to the proof of the LDP for $L^N(\uX^{\underline{\hat\alpha}})$, it remains to show that the sequences $(L^N(\underline{ \alpha}_t))_N$ and $(L^N(\underline{\hat\alpha}_t))_N$ are exponentially close and that the rate function $\mathcal{\widetilde I}$ has compact sublevel sets.
	The latter property follows from the fact that for every $x\ge 0$ it holds $\{ \nu\in\cP(\mathbb{R}^m): \mathcal{\widetilde I}(\nu)\le x\} = \Psi\big(t,\{ \theta_t: \theta\in \cP_2(\cC^\ell),\,\,  \mathcal{I}(\theta)\le x\} \big) $ and $\Psi(t,\cdot)$ is a continuous function for the weak topology.
	Let us now show exponential closeness.
	To this end, we introduce the function $\Phi$ mapping $\cP(\RR^\xdim\times \RR^\xdim)\times \cP(\RR^m)$ to $\cP(\RR^m)$ and defined as
	\begin{equation}
		\Phi( t,\xi,\mu) := \xi\otimes\mu \circ \widehat\Lambda\big(t, \cdot,\cdot ,\xi^1,\cdot \big)^{-1} 
	\end{equation}
	where $\xi^1$ is the first marginal of $\xi$. 
	That is, for a Borel set $U \subseteq \RR^m$, we have $\Phi(t, \xi, \mu)(U) = \xi\otimes \mu(\{(x,y,z): \widehat \Lambda(t, x, y, \xi^1, z)\in U\})$. 
	Recall that we put $\Lambda(t,x,y,\mu) := \widehat\Lambda(t,x,y,\mu,0)$.	
	Now, Put
	\begin{equation*}
	 	\cY_t^{i,N} := V(t, \cX_t^{i,N}, L^N(\underline{\cX}_t) ).
	 \end{equation*} 
	 By the definition of $\Phi$, and $\Phi^N$ we have
	$$L^N(\underline{\alpha}_t) = \Phi(t, L^N(\underline{\cX}_t, \underline{\cY}_t),\delta_0) \quad \text{and} \quad L^N(\underline{\hat\alpha}_t) = \Phi(t, L^N(\underline{X}^{\underline{\hat\alpha}}_t, \underline{Y}_t),L^N(\underline{\zeta}_t)) $$ 
	with $\underline{\zeta}_t:= (\zeta^{1,N}_t, \dots ,\zeta^{N,N}_t)$.
	show that the function $\Phi$ is is Lipschitz continuous with respect to the $2$-Wasserstein topology.
	By Kantorovich's duality theorem, see \cite[Theorem 5.10]{Vil2}, for every $\xi,\xi' \in \mathcal{P}_2(\mathbb{R}^\xdim\times \RR^\xdim)$ we have
	\begin{align*}
		\cW^2_2(\Phi(t,\xi,\mu), \Phi(t,\xi',\mu')) &= \sup\bigg(\int_{\RR^m} h_1(x)\Phi(t,\xi,\mu)(dx) - \int_{\RR^m} h_2(x')\Phi(t,\xi',\mu')(dx') \bigg)\\
		& = \sup\bigg(\int_{\RR^\xdim\times \RR^\xdim\times \RR^m} h_1(\widehat\Lambda(t,x,y, \xi^1,z))\xi(dx,dy)\mu(dz))\\
		&\qquad
	 - \int_{\RR^\xdim\times \RR^\xdim\times \RR^m} h_2(\widehat\Lambda(t, x', y',\xi^{1\prime}, z))\xi'(dx',dy')\mu'(dz) \bigg)
	\end{align*}
	with the supremum being taken over the set of bounded continuous functions $h_1,h_2:\RR^m\to \RR$ such that $h_1(x) -h_2(x') \le |x-x'|^2 $ for every $x, x'\in \mathbb{R}^m$ which, by Lipschitz continuity of $\widehat\Lambda$ implies that $h_1\Big(\widehat\Lambda(t, x, y, \xi^1,z) \Big) - h_2\Big( \widehat{\Lambda}(t, x',y', \xi^{1\prime}, z') \Big)\le C(|x -x'|^2 +|y-y'|^2 +\cW^2_2(\xi^1,\xi^{1\prime}) + |z-z'|^2)$ for some constant $C>0$.
	This shows that
	\begin{align*}
		\cW^2_2(\Phi(t,\xi,\mu), \Phi(t,\xi',\mu'))  &\le C\sup\bigg( \int_{\mathbb{R}^\xdim\times \RR^\xdim\times \RR^m} \tilde h_1(x,y,z)\xi(dx,dy)\mu(dz)\\
		&\qquad - \int_{\mathbb{R}^\xdim\times \RR^\xdim\times \RR^m} \tilde h_2(x',y',z')\xi'(dx',dy')\mu'(dz) \bigg)
		  +C\cW^2_2(\xi^1,\xi^{1\prime}) 
	\end{align*}
	with the supremum over functions $\tilde h_1,\tilde h_2$ such that $\tilde h_1(x,y,z) - \tilde h_2(x',y',z') \le |x-x'|^2 + |y- y'|^2 + |z - z'|^2$.
	Hence, applying Kantorovich duality once again yields
	\begin{align}
	\label{eq:psi.Lipschitz}
		\cW_2(\Phi(t,\xi,\mu), \Phi(t,\xi',\mu')) \le C\cW_2(\xi\otimes \mu, \xi'\otimes \mu')+ C\cW_2(\xi^1,\xi^{1\prime}). 
	\end{align}
	Thus, we have
	\begin{align*}
		&\PP\Big(\sup_{t\in [0,T]}\cW_2\big(L^N(\underline{\hat\alpha}_t), L^N(\underline{\alpha}_t) \big)> \varepsilon \Big)= \PP\Big(\sup_{t\in [0,T]}\cW_2\big(\Phi(t,L^N(\underline{X}^{\underline{\alpha}}_t,\underline{Y}_t),\delta_0),\Phi(t, L^N(\underline{\cX}_t, \underline{\cY}_t),L^N(\underline{\zeta}_t))\big) > \varepsilon \Big) \\
		&\le \PP\Big(\sup_{t\in [0,T]}\cW_2( L^N(\underline{X}^{\underline{\alpha}}_t,\underline{Y}_t)\otimes \delta_0, L^N(\underline{\cX}_t, \underline{\cY}_t))\otimes L^N(\underline{\zeta}_t) ) >\frac{\varepsilon}{2C} \Big) + \PP\Big(\sup_{t\in [0,T]}\cW_2( L^N(\underline{X}^{\underline{\alpha}}_t), L^N(\underline{\cX}_t))>\frac{\varepsilon}{2C} \Big)\\
		&\le  Ce^{-\varepsilon^2N^2} + \EE\Big[\exp\Big(\sum_{i=1}^N|\zeta^{i,N}|\Big)^2\Big]e^{-\varepsilon^2N^2C}\\
		&\le Ce^{-\varepsilon^2N^2}
	\end{align*}
	for some constant $C>0$, where the latter inequality follows by Lemma \ref{lem:bound_Y}.
	Therefore, it follows that	
	\begin{equation*}
	 	\lim_{N\to \infty}\frac1N\log \PP\Big(\sup_{t\in [0,T]}\cW_2(L^N(\underline{\hat\alpha}_t), L^N(\underline{\alpha_t})) >\varepsilon \Big) = - \infty.
	 \end{equation*}
	This concludes the proof.
\end{proof}
\begin{proof}[Proof of Theorem \ref{thm:LDP-Tlarge}]
	The proof of this theorem is almost the same as that of Theorem \ref{thm:LDP}, except for two points.
	First, to get well-posedness of the FBSDEs \eqref{eq:aux.fbsde} and \eqref{eq:MckV fbsde}, we rather use the results of \citet{Peng-Wu99} and \citet{Ben-Yam-Zhang15} respectively.
	Secondly, one should apply the second part of Lemma \ref{lem:bound_Y} below, (i.e.\ the case of $T$ arbitrary) to get the bound of $\sum_{i=1}^N\|X^{\underline{\alpha},i} -  \cX^{i}\|^2_\infty$.
\end{proof}
We finish the proof with the following key lemma used to derive the large deviation principles. 
We will denote by $\mathcal{S}^{\infty}(\mathbb{R}^l)$ the space of continuous adapted processes $X$ such that  
\begin{equation*}
	\|X\|_{\mathcal{S}^{\infty}(\mathbb{R}^l)}:= \sup_{0\leq t\leq T}|X_t|\in \mathbb{L}^{\infty}(\mathbb{R},\cF)<\infty
\end{equation*}
where $\mathbb{L}^\infty(\mathbb{R},\cF)$ is the space of essentially bounded $\cF$--measurable random variables with values in $\RR$.
\begin{lemma}\label{lem:bound_Y}
	If the assumptions \ref{a1}-\ref{a6} are satisfied, then it holds that
	\begin{equation}
	\label{eq:bound.sum}
		\sup_{N\ge1}\Big\|\sum_{j=1}^N|Y^{i,j}|\Big\|_{\mathcal{S}^\infty(\mathbb{R}^\ell)} <   \infty \quad \text{for all}\quad t \in [0,T]\quad \text{and for each } i\ge1
	\end{equation}
	and there is a constant $C>0$ such that
	\begin{equation}
	\label{eq:eps.zeta.gamma}
		|\varepsilon^{i,N}| + |\gamma^{i,N}| + |\zeta^{i,N}| \le C/N. 
	\end{equation}
	Moreover, there is $c>0$ such that if $T\le c$, then
	\begin{equation}
	\label{eq:bound.X.Y}
		\sup_{N\ge1}\Big\|\sum_{i=1}^N|X^{i,\underline{\hat{\alpha}}} - \cX^{i,N}|\Big\|_{\mathcal{S}^{\infty}(\mathbb{R}^\ell)} + \sup_{N\ge1}\Big\|\sum_{i=1}^N|Y^{i,i}-V(\cdot,\cX^{i,N}_\cdot,L^N(\underline\cX_\cdot))|\Big\|_{\mathcal{S}^{\infty}(\mathbb{R}^\ell)}< \infty.
	\end{equation}
	
	If the conditions \ref{a1}-\ref{a8} hold, then for abitrarily large $T>0$, there is a constant $c(T,L_{b,a},L_{b,\xi},L_\Lambda)$ such that if $K_b> c(T,L_{b,a},L_{b,\xi},L_\Lambda)$, then the bound \eqref{eq:bound.X.Y} holds.
\end{lemma}
\begin{proof}
	Taking the conditional expectation in \eqref{eq:Y.ij} with respect to $\mathcal{F}^N_t$ and recalling the definition of $H^{N,i}$, we get
	\begin{align*}
		|Y^{i,j}_t|&\leq \EE\bigg[\Big|\delta_{\{i=j\}}\partial_xg\Big(X^{i,\underline{\hat{\alpha}}}_T,L^N(\underline{X}^{\underline{\hat{\alpha}}}_T)\Big)\Big|+\frac{1}{N}\Big|\partial_{\mu}g\Big(X^{i,\underline{\hat{\alpha}}}_T,L^N(\underline{X}^{\underline{\hat{\alpha}}}_T)\Big)(X^{j,\underline{\hat{\alpha}}}_T)\Big|\Big|\mathcal{F}^N_t\bigg]\\
		& + \EE\bigg[\int_t^T\Big(\Big|\delta_{\{i=j\}}\partial_{x}f\Big(s,X^{i,\underline{\hat{\alpha}}}_s,\hat{\alpha}^i_s,L^N(\underline{X}^{\underline{\hat{\alpha}}}_s,\underline{\hat{\alpha}}_s)\Big)\Big|+\Big|\frac{1}{N}\partial_{\mu}f\Big(s,X^{i,\underline{\hat{\alpha}}}_s,\hat{\alpha}^i_s,L^N(\underline{X}^{\underline{\hat{\alpha}}}_s,\underline{\hat{\alpha}}_s)\Big)(X^{j,\underline{\hat{\alpha}}}_s)\Big|\Big)ds\Big|\mathcal{F}^N_t\bigg]\\
		& + \EE\bigg[\int_t^T\Big(\Big|\partial_xb\Big(s,X^{j,\underline{\hat{\alpha}}}_s,\hat{\alpha}^j_s,L^N(\underline{X}^{\underline{\hat{\alpha}}}_s,\underline{\hat{\alpha}}_s)\Big)Y^{i,j}_s\Big|+\frac{1}{N}\sum_{k=1}^N\Big|\partial_\mu b\Big(s,X^{k,\underline{\hat{\alpha}}}_s,\hat{\alpha}^k_s,L^N(\underline{X}^{\underline{\hat{\alpha}}}_s,\underline{\hat{\alpha}}_s)\Big)(X^{j,\underline{\hat{\alpha}}}_s)Y^{i,k}_s\Big|\Big)ds\Big|\mathcal{F}^N_t\bigg].
	\end{align*}
	Next, using boundedness of the functions $\partial_xg$, $\partial_{\mu}g$, $\partial_xf$, $\partial_{\mu}f$, $\partial_xb$ and $\partial_{\mu} b$ yields
	\begin{align*}
		|Y^{i,j}_t|&\leq C_1\delta_{\{i=j\}}+ C_2/N+C_3\EE\bigg[\int_t^T\Big(|Y^{i,j}_s|+\frac{1}{N}\sum_{k=1}^N|Y^{i,k}_s|\Big)ds\Big|\mathcal{F}^N_t\bigg].
	\end{align*}
	Summing up over $j$ and using Gronwall's inequality yields \eqref{eq:bound.sum}.

	Let us now show \eqref{eq:eps.zeta.gamma}.
	This follows directly from the above bound and the fact that the measure derivatives of the coefficients $b,f$ and $g$ are bounded.
	In fact, it follows by definition of $\varepsilon^N,\gamma^N$ and $\zeta^N$ that
	\begin{equation*}
		|\varepsilon^N| + |\gamma^N| + |\zeta^N| \le C_1\frac1N + C_2\frac1N\sum_{j=1}^N|Y^{i,j}_t|\le C/N,
	\end{equation*}
	where the second equality follows by \eqref{eq:bound.sum}.

	We will prove the bound \eqref{eq:bound.X.Y} only when the monotonicity condition \ref{a8} is satisfied.
	The case $T$ small follows by similar arguments, we will explain the difference at the end of the proof.
	Thus, we start by denoting
	$$\cY^{i,N}_t := V(t,\cX^{i,N}_t,L^N(\underline\cX_t)).$$ 
	Recall the function $\bar B$ introduced in \eqref{eq:def.barB} and the auxiliary control 	
	\begin{equation*}
		\alpha^{i,N}_t := \Lambda\Big(t,\cX^{i,N}_t, V\big(t,\cX^{i,N}_t, L^N(\underline{\cX}_t)\big),L^N(\underline{ \cX}_t) \Big).
	\end{equation*}
	By Lipschitz--continuity of the functions $b,\Lambda$ and\footnote{Without loss of any generality, we denote the Lipschitz constant of $\widehat{\Lambda}$ by $L_{\Lambda}$.} $\widehat{\Lambda}$ and recalling the representation \eqref{eq:N Nash} of the Nash equilibrium $\hat\alpha^{i,N}$, we have
	\begin{align*}
		&|X^{i,\underline{\hat{\alpha}}}_t - \cX^{i,N}_t|^2
		= 2\int_0^t(X^{i,\underline{\hat{\alpha}}}_s - \cX^{i,N}_s)\Big(b\big(s,X^{i,\underline{\hat{\alpha}}}_s,\hat{\alpha}^{i,N}_s,L^N(\underline{X}^{\underline{\hat{\alpha}}}_s,\underline{\hat{\alpha}}_s) \big) - \overline B\big(s,\cX^{i,N}_s,L^N(\underline{\cX}_s) \big) \Big)ds\\
		&= 2\int_0^t(X^{i,\underline{\hat{\alpha}}}_s - \cX^{i,N}_s)\bigg(b\big(s,X^{i,\underline{\hat{\alpha}}}_s,\hat{\alpha}^{i,N}_s,L^N(\underline{X}^{\underline{\hat{\alpha}}}_s,\underline{\hat{\alpha}}_s) \big)\\
		&\qquad - b\big(s,\cX^{i,N}_s,\Lambda(s,\cX^{i,N}_s,\cY^{i,N}_s,L^N(\underline\cX_s)),L^N(\underline\cX_s,\underline{\alpha}_s)) \big) \bigg)ds\\
		&\leq 2\int_0^t|X^{i,\underline{\hat{\alpha}}}_s - \cX^{i,N}_s|\bigg(-K_b|X^{i,\underline{\hat{\alpha}}}_s - \cX^{i,N}_s|+ L_{b,a}L_\Lambda|Y^{i,i}_s-\cY^{i,N}_s|+L_{b,a}L_\Lambda\mathcal{W}_2(L^N(\underline{X}^{\underline{\hat{\alpha}}}_s),L^N(\underline{\cX}_s))\\
		&\qquad+L_{b,a}L_\Lambda|\zeta^{i,N}_s|+L_{b,\xi}\mathcal{W}_2(L^N(\underline{X}^{\underline{\hat{\alpha}}}_s,\underline{\hat{\alpha}}_s),L^N(\underline{\cX}_s,\underline{\alpha}_s))\bigg)ds\\
		&\leq 2\int_0^t|X^{i,\underline{\hat{\alpha}}}_s - \cX^{i,N}_s|\bigg(-K_b|X^{i,\underline{\hat{\alpha}}}_s - \cX^{i,N}_s|+L_{b,a}L_\Lambda|Y^{i,i}_s-\cY^{i,N}_s|+ (L_{b,a}L_\Lambda+L_{b,\xi})\Big\{\frac{1}{N}\sum_{j=1}^N|X^{j,\underline{\hat{\alpha}}}_s-\cX^{j,N}_s|^2\Big\}^{1/2}\\
		&\qquad+4L_\Lambda L_{b,\xi}\Big\{\frac{1}{N}\sum_{j=1}^N\Big(|X^{j,\underline{\hat{\alpha}}}_s - \cX^{j,N}_s|^2+|Y^{j,j}_s-\cY^{j,N}_s|^2 + \frac{1}{N}\sum_{j=1}^N|X^{j,\underline{\hat{\alpha}}}_s-\cX^{j,N}_s|^2 + |\zeta^{j,N}_s|^2\Big)\Big\}^{1/2}\\
		&\qquad +L_{b,a}L_\Lambda|\zeta^{i,N}_s|\bigg)ds\\
		&\leq 2\int_0^t|X^{i,\underline{\hat{\alpha}}}_s - \cX^{i,N}_s|\bigg(-K_b|X^{i,\underline{\hat{\alpha}}}_s - \cX^{i,N}_s|+L_{b,a}L_\Lambda|Y^{i,i}_s-\cY^{i,N}_s|+ (L_{b,a}L_\Lambda+L_{b,\xi})\Big\{\frac{1}{N}\sum_{j=1}^N|X^{j,\underline{\hat{\alpha}}}_s-\cX^{j,N}_s|^2\Big\}^{1/2}\\
		&\qquad+8L_{b,\xi}L_\Lambda\Big\{\frac{1}{N}\sum_{j=1}^N|X^{j,\underline{\hat{\alpha}}}_s - \cX^{j,N}_s|^2\Big\}^{1/2}+4L_{b,\xi}L_\Lambda\Big\{\frac{1}{N}\sum_{j=1}^N|Y^{j,j}_s-\cY^{j,N}_s|^2\Big\}^{1/2}\\
		&\qquad+4L_{b,\xi}L_\Lambda\Big\{\frac{1}{N}\sum_{j=1}^N|\zeta^{j,N}_s|^2\Big\}^{1/2} + L_{b,a}L_\Lambda|\zeta^{i,N}_s|\bigg)ds.
	\end{align*}
	Applying Young's inequality for some $\varepsilon>0$ and $\eta>0$ to be determined, we continue the estimation as
	\begin{align*}
		|X^{i,\underline{\hat{\alpha}}}_t - \cX^{i,N}_t|^2 & \leq 2\int_0^t\bigg(\Big({\frac{L_{b,\xi}+L_{b,a}L_\Lambda+8L_{b,\xi}L_{\Lambda}}{2}+\frac{L_{b,a}^2L_\Lambda^2+16L_{b,\xi}^2L_{\Lambda}^2}{2\varepsilon}+\eta-K_b}\Big)|X^{i,\underline{\hat{\alpha}}}_s - \cX^{i,N}_s|^2\\
		&\qquad+\frac{\varepsilon}{2}|Y^{i,i}_s-\cY^{i,N}_s|^2+ \frac{{L_{b,\xi}+L_{b,a}L_\Lambda+8L_{b,\xi}L_{\Lambda}}}{2}\frac{1}{N}\sum_{j=1}^N|X^{j,\underline{\hat{\alpha}}}_s-\cX^{j,N}_s|^2\\
		&\qquad+\frac{L_{b,a}^2L_\Lambda^2}{2\eta}|\zeta^{i,N}_s|^2+\frac{\varepsilon}{2}\frac{1}{N}\sum_{j=1}^N|Y^{j,j}_s-\cY^{j,N}_s|^2+\frac{8{L_{b,\xi}^2}L_\Lambda^2}{\eta}\frac{1}{N}\sum_{j=1}^N|\zeta^{j,N}_s|^2\bigg)ds.
	\end{align*}
	Let us introduce the quantity
	\begin{equation*}
       \delta(\varepsilon,\eta)= L_{b,\xi}+L_{b,a}L_\Lambda+8L_{b,\xi}L_{\Lambda}+\frac{L_{b,a}^2L_\Lambda^2+16L_{b,\xi}^2L_{\Lambda}^2}{2\varepsilon}+\eta-K_b.
    \end{equation*}
	Summing up over $i$ on both sides and using Gronwall's inequality implies that
	\begin{equation}
		\label{eq:bound.on.diff.X_mon}
			\sum_{i=1}^N|X^{i,\underline{\hat{\alpha}}}_t - \cX^{i,N}_t|^2\leq 2e^{2\delta(\varepsilon,\eta)t}\int_0^T\Big(\varepsilon\sum_{i=1}^N|Y^{i,i}_s-\cY^{i,N}_s|^2+\frac{L_{b,a}^2L_\Lambda^2+16L_{b,\xi}^2L_{\Lambda}^2}{2\eta}\sum_{i=1}^N|\zeta^{i,N}_s|^2\Big)ds.
	\end{equation}
	On the other hand, since $V$ is a classical solution of the PDE \eqref{eq:master pde},
	applying It\^{o}'s formula for functions of the law of a diffusion (see \cite[Theorem 5.104]{MR3752669}) implies that $\cY^{i,N}$ satisfies
	\begin{equation*}
		\cY^{i,N}_t=\partial_xg(\cX^{i,N}_T,L^N(\underline \cX_T))+\int_t^TF(s,\cX^{i,N}_s,\cY^{i,N}_s,L^N(\underline\cX_s,\underline\cY_s))ds-\int_t^T\sum_{k=1}^{N}\cZ^{i,N,k}_sdW^k_s
	\end{equation*}
	for an $(\cF^N_t)$--progressive process $\cZ^{i,k,N}$ given by 
	$$\cZ^{i,k,N}_t:=\partial_{x^k}V(t, \mathcal{X}^i_t, L^N(\underline{\cX}_t)).$$ 
	Using the Lipschitz--continuity condition on  $\partial_x H(t,x,y,a,\xi)$ and $\partial_xg(x,\mu)$ (see \ref{a3}), it follows by \eqref{eq:true.system.N} and \eqref{eq:aux.fbsde} that
	\begin{align*}
		&|Y^{i,i}_t-\cY^{i,N}_t|^2\\
		&= \EE\bigg[\Big|\partial_xg(X^{i,\underline{\hat{\alpha}}}_T,L^N(\underline{X}^{\underline{\hat{\alpha}}}_T))+\gamma^{i,N}-\partial_xg(\cX^{i,N}_T,L^N(\underline{\cX_T}))\Big|^2 \Big| \cF^N_t\bigg]\\
		&\quad+ \EE\bigg[2\int_t^T\big(Y^{i,i}_s-\cY^{i,N}_s\big)\bigg(\partial_xf\Big(s,X^{i,\underline{\hat{\alpha}}}_s,\hat{\alpha}^{i,N}_s,L^N(\underline{X}^{\hat{\alpha}}_s,\underline{\hat{\alpha}}_s)\Big)+\partial_xb\Big(s,X^{i,\underline{\hat{\alpha}}}_s,\hat{\alpha}^{i,N}_s,L^N(\underline{X}^{\hat{\alpha}}_s,\underline{\hat{\alpha}}_s)\Big)Y^{i,i}_s+\varepsilon^{i,N}_s\\
		&\qquad-\partial_xf\Big(s,\cX^{i,N}_s,\Lambda\big(s,\cX^{i,N}_s,\cY^{i,N}_s,L^N(\underline{\cX}_s)\big),L^N(\underline{\cX}_s,\underline{\alpha}_s)\Big)\\
		&\qquad-\partial_xb\Big(s,\cX^{i,N}_s,\Lambda\big(s,\cX^{i,N}_s,\cY^{i,N}_s,L^N(\underline{\cX}_s)\big),L^N(\underline{\cX}_s,\underline{\alpha}_s)\Big)\cY^{i,N}_s\bigg)ds \Big| \mathcal{F}^N_t\bigg]\\
		&\leq \EE\bigg[3L_g^2|X^{i,\underline{\hat{\alpha}}}_T-\cX^{i,N}_T|^2+\frac{3L_g^2}{N}\sum_{j=1}^N|X^{j,\underline{\hat{\alpha}}}_T - \cX^{j,N}_T|^2+3|\gamma^{i,N}|^2\bigg| \cF^N_t\bigg]\\
		&\quad+ \EE\bigg[2\int_t^T|Y^{i,i}_s-\cY^{i,N}_s|\bigg(
		L_f|X^{i,\underline{\hat{\alpha}}}_s-\cX^{i,N}_s|+L_f|Y^{i,i}_s-\cY^{i,N}_s|+|\varepsilon^{i,N}_s|\\
		&\qquad+L_f\Big|\hat{\alpha}^{i,N}_s-\Lambda\big(s,\cX^{i,N}_s,\cY^{i,N}_s,L^N(\underline{\cX}_s)\big)\Big|+ L_f\mathcal{W}_2\Big(L^N(\underline{X}^{\underline{\hat{\alpha}}}_s,\underline{\hat{\alpha}}_s),L^N(\underline{\cX}_s,\underline{\alpha}_s)\Big)ds \bigg)\Big| \mathcal{F}^N_t\bigg]\\
		&\leq \EE\bigg[L_g^2|X^{i,\underline{\hat{\alpha}}}_T-\cX^{i,N}_T|^2+\frac{3L_g^2}{N}\sum_{j=1}^N|X^{j,\underline{\hat{\alpha}}}_T - \cX^{j,N}_T|^2 + 3|\gamma^{i,N}|^2\bigg|\cF^N_t\bigg]\\
		&\quad+\EE\bigg[2\int_t^T|Y^{i,i}_s-\cY^{i,N}_s|\bigg(
		L_f(1+L_\Lambda)\big(|X^{i,\underline{\hat{\alpha}}}_s-\cX^{i,N}_s|+|Y^{i,i}_s-\cY^{i,N}_s| \big) +|\varepsilon^{i,N}_s|+L_fL_\Lambda|\zeta^{i,N}_s|\\
		&\quad+L_f\Big\{ \frac{1}{N}\sum_{j=1}^N\Big((1+4L^2_\Lambda)|X^{j,\underline{\hat{\alpha}}}_s-\cX^{j,N}_s|^2+4L_\Lambda^2|Y^{j,j}_s-\cY^{j,N}_s|^2+4L_\Lambda^2|\zeta^{j,N}_s|^2 +4L_\Lambda^2\frac{1}{N}\sum_{j=1}^N |X^{j,\underline{\hat{\alpha}}}_s-\cX^{j,N}_s|^2\Big) \Big\}^{1/2}\\
		&\quad + L_fL_\Lambda\Big\{\frac{1}{N}\sum_{j=1}^N |X^{j,\underline{\hat{\alpha}}}_s-\cX^{j,N}_s|^2\Big\}^{1/2} \bigg)\, ds \Big| \mathcal{F}^N_t\bigg]\\
		&\leq \EE\bigg[3L_g^2|X^{i,\underline{\hat{\alpha}}}_T-\cX^{i,N}_T|^2+3L_g^2\Big\{\frac{1}{N}\sum_{j=1}^N|X^{j,\underline{\hat{\alpha}}}_T - \cX^{j,N}_T|^2\Big\}+3|\gamma^{i,N}|^2\bigg| \cF^N_t\bigg]\\
		&\quad+ \EE\bigg[2\int_t^T|Y^{i,i}_s-\cY^{i,N}_s|\bigg(
		L_f(1+L_\Lambda )\big(|X^{i,\underline{\hat{\alpha}}}_s-\cX^{i,N}_s|+|Y^{i,i}_s-\cY^{i,N}_s|\big)+|\varepsilon^{i,N}_s|+L_fL_\Lambda|\zeta^{i,N}_s|\\
		&\quad +\Big(L_f(1 + L_\Lambda)\sqrt{1+8L_\Lambda^2}\Big)\Big\{\frac{1}{N}\sum_{j=1}^N |X^{j,\underline{\hat{\alpha}}}_s-\cX^{j,N}_s|^2\Big\}^{1/2}\\
		&\quad+2L_fL_\Lambda\Big\{ \frac{1}{N}\sum_{j=1}^N|Y^{j,j}_s-\cY^{j,N}_s|^2\Big\}^{1/2}+2L_fL_\Lambda\Big\{ \frac{1}{N}\sum_{j=1}^N|\zeta^{j,N}_s|^2\Big\}^{1/2}\bigg)\, ds \Big| \mathcal{F}^N_t\bigg].
	\end{align*}
	Applying Young's inequality for some constants $\varepsilon_1>0$ and $\eta_1>0$ to be determined, we continue the estimations as
	\begin{align*}
	|Y^{i,i}_t-\cY^{i,N}_t|^2
		&\leq \EE\bigg[3L_g^2|X^{i,\underline{\hat{\alpha}}}_T-\cX^{i,N}_T|^2+3L_g^2\Big\{\frac{1}{N}\sum_{j=1}^N|X^{j,\underline{\hat{\alpha}}}_T - \cX^{j,N}_T|^2\Big\}+3|\gamma^{i,N}|^2\bigg| \cF_t^N\bigg]\\
		&\qquad+ \EE\bigg[ \int_t^T\bigg( (2L_f+2L_\Lambda L_f+3\varepsilon_1+3\eta_1)|Y^{i,i}_s-\cY^{i,N}_s|^2+\frac{(L_f+L_fL_\Lambda)^2}{\varepsilon_1}|X^{i,\underline{\hat{\alpha}}}_s-\cX^{i,N}_s|^2\\
		&\qquad +\frac{(L_\Lambda L_f+ L_f\sqrt{1+8L_\Lambda^2})^2}{\varepsilon_1}\frac{1}{N}\sum_{j=1}^N |X^{j,\underline{\hat{\alpha}}}_s-\cX^{j,N}_s|^2+\frac{4L_f^2L_\Lambda^2}{\varepsilon_1} \frac{1}{N}\sum_{j=1}^N|Y^{j,j}_s-\cY^{j,N}_s|^2\\
		&\qquad+\frac{3}{\eta_1}|\varepsilon^{i,N}_s|^2+\frac{3L_f^2L_\Lambda^2}{\eta_1}|\zeta^{i,N}_s|^2+\frac{6L_\Lambda^2 L_f^2}{\eta_1} \frac{1}{N}\sum_{j=1}^N|\zeta^{j,N}_s|^2 \bigg)\, ds \Big| \mathcal{F}^N_t\bigg],
	\end{align*}
	where we also used the representation \eqref{eq:N Nash} of $\hat\alpha^{i,N}$ and the fact that $\Lambda$ is Lipschitz--continuous.
	Denoting
	\begin{equation*}
        \delta(\varepsilon_1,\eta_1) :=2L_f+2L_\Lambda L_f+3\varepsilon_1+3\eta_1+\frac{4L_\Lambda^2L_f^2}{\varepsilon_1}
    \end{equation*}
	summing up on both sides and applying Gronwall inequality, we obtain
	\begin{align*}
		&\sum_{i=1}^N|Y^{i,i}_t-\cY^{i,N}_t|^2 \\
		&\leq e^{\delta(\varepsilon_1,\eta_1)(T-t)}\bigg\{ 
			6L_f^2+\Big(\frac{L_f(1+L_\Lambda)^2}{\varepsilon_1}+\frac{L_\Lambda(L_f+\sqrt{1+8L_\Lambda^2})^2}{\varepsilon_1}  \Big)T\bigg\}\Big\|\sum_{i=1}^N|X^{i,\underline{\hat{\alpha}}} - \cX^{i,N}|^2\Big\|_{\mathcal{S}^{\infty}}\\
			&\quad+3e^{\delta(\varepsilon_1,\eta_1)(T-t)}\bigg\{\sum_{i=1}^N\|\gamma^{i,N}\|_{L^{\infty}}^2+\frac{1}{\eta_1}\sum_{i=1}^N\|\varepsilon^{i,N}\|_{\mathcal{S}^{\infty}}^2+3\frac{L_\Lambda^2L_f^2}{\eta_1}\sum_{i=1}^N\|\zeta^{i,N}\|_{\mathcal{S}^{\infty}}^2
		\bigg\}.
	\end{align*}
	If $K_b$ is large enough, then $\delta(\varepsilon, \eta)\le 0$.
	In this case, combining this with \eqref{eq:bound.on.diff.X_mon}, it follows that we can choose $\varepsilon$ small enough that
	\begin{equation}
	\label{eq:cond.Tvseps}
		2\varepsilon Te^{2\delta(\varepsilon,\eta)T+\delta(\varepsilon_1,\eta_1)T}\bigg\{ 
			6L_f^2+\Big(\frac{L_f(1+L_\Lambda)^2}{\varepsilon_1}+\frac{L_\Lambda(L_f+\sqrt{1+8L_\Lambda^2})^2}{\varepsilon_1}  \Big)T\bigg\} < 1.
	\end{equation}
	Therefore, we have
	\begin{equation*}
		\Big\|\sum_{i=1}^N|Y^{i,i}-\cY^{i,N}|\Big\|_{\mathcal{S}^{\infty}(\mathbb{R}^\xdim)}\leq C\sum_{i=1}^N\Big(\|\gamma^{i,N}\|_{L^{\infty}}+\|\varepsilon^{i,N}\|_{\mathcal{S}^{\infty}}+\|\zeta^{i,N}\|_{\mathcal{S}^{\infty}}\Big)
	\end{equation*}
	for a constant $C>0$.
	Since the bound of $\gamma^{i,N},\varepsilon^{i,N},\zeta^{i,N}$ is $O(N^{-1})$, we obtain \eqref{eq:bound.X.Y} in view of \eqref{eq:bound.on.diff.X_mon}.

	When the monotonicity condition \ref{a8} is not assumed, we can use the same argument (with $K_b=-L_{b}-L_{b,a}L_{\Lambda}$) and in this case we need $T$ small enough to get \eqref{eq:cond.Tvseps}.
\end{proof}

\begin{remark}
	\begin{itemize}
		\item[(i)]Observe that in \ref{a8}, the condition \eqref{eq:mon.con.b.H.g} is needed only to guarantee existence for the FBSDE \eqref{eq:aux.fbsde} for arbitrary time horizons.
		When these FBSDEs are known to have solutions only the monotonicity condition \eqref{eq:mon.con.b} on $b$ is needed for the LDP.
		\item[(ii)]It is easily checked that when the functions $\partial_xf, \partial_xg, \partial_\mu f$ and $\partial_\mu g$ are of linear growth the conclusion of Lemma \ref{lem:bound_Y} remains true for $T$ small enough. The boundedness conditions are needed for the extension of the arguments to arbitrarily large time horizons.
	\end{itemize}
\end{remark}

\subsection{Example: A model of systemic risk}
\label{sec:example}
	For illustration, consider (as slight modification of) the problem of inter-bank borrowing and lending studied by \citet{MR3325083}. 
In the systemic risk model proposed by these authors, bank i's reserve is given by the dynamics
\begin{equation*}
	dX^i_t =a\Big(\frac1N\sum_{j=1}^NX^j_t -X^i_t \Big) + \alpha^i_t\, dt + \sigma \,dW^i_t
\end{equation*}
for a given mean-reverting parameter $a\ge0$.
Each bank controls its borrowing and lending rate $\alpha^i$ at time $t$ by choosing it so as to minimize the cost function
\begin{equation*}
	J^i(\underline \alpha) = \EE\Big[\int_0^Tf(X_t^i, L^N(\underline X_t), \alpha^i_t)\,dt + g(X^i_T) \Big]
\end{equation*}
with 
\begin{equation*}
	f(x,\mu,\alpha):= \frac12\alpha^2 -q\alpha\Big(\int_{\mathbb{R}}z\mu(dz) - x\Big) + \frac{\varepsilon}{2}\Big|\int_{\mathbb{R}}z\mu(dz) - x\Big| \quad\text{and}\quad g(x) = \frac{c}{2}\Big|\int_{\mathbb{R}}z\mu(dz) - x \Big|
\end{equation*}
where $q,c$ and $\varepsilon$ are strictly positive parameters such that $q^2\le \varepsilon$.
With these specifications of the coefficients of the game, the conditions \ref{a1}-\ref{a5} are clearly satisfied.
To check \ref{a6}, notice that the function $\Lambda$ therein now takes the form
\begin{equation*}
	\Lambda(x,\mu,y) = q\Big(\int_{\mathbb{R}}z\mu(dz) - x\Big) - y.
\end{equation*}
In particular, the functions $B,F$ and $G$ in \ref{a6} are linear and time-independent.
Thus, it follows from \cite[Proposition 5.2]{ChassagneuxCrisanDelarue_Master} and \cite[Lemma 5.6]{Car-Del15}  that the PDE \eqref{eq:master pde} admits a solution $V$ which is Lipschitz--continuous in its second and third variables.
Therefore, if for each $N$ the finite player game admits a Nash equilibrium, (see e.g. \cite[Section 3.1]{MR3325083} for details on the existence) $\underline{\hat\alpha} = (\hat\alpha^{1,N},\dots\hat\alpha^{N,N})$ then it follows by Theorem \ref{thm:LDP} that if $T$ is small enough, then the sequence $(L^N(\underline{X}^{\underline{\hat\alpha}}))$ satisfies the LDP with rate function given by \eqref{eq:rate.X}, wherein
\begin{equation*}
	\overline B(t,x, \mu) := (a+ q) \Big( \int_{\mathbb{R}}z\mu(dz) - x \Big) - V(t, x,\mu).
\end{equation*}
For LDP with $T$ arbitrary, by Theorem \ref{thm:LDP-Tlarge}, one needs additional monotonicity properties. 
These can be guaranteed by appropriate conditions on the constants $\varepsilon, q, c$ $a$ and $T$.
Note that 
the large deviation principle for this example was obtained in \cite{Del-Lac-Ram_Concent} for the case of closed-loop controls.

\section{Large deviation principle for cooperative games: a case study}
\label{sec:ldp-mfc}
As explained in the introduction, the method developed in this article also applies to cooperative large population games.
Such games are important in several applications, especially when a ``societal goal'' should be achieved.
We refer the reader for instance to \cite{Carmona-Delarue12} and \cite{Gob-Gran19}.
In the general setting, and keeping the notation of the previous section, a continuous time differential cooperative game takes the form
\begin{equation}
\label{eq:Coop.game}
	V^N := \inf_{\ua \in \mathcal{A}^N}\frac1N\sum_{i=1}^N\EE\Big[\int_0^T f(t, X^{i,\ua}_t, {\alpha^i_t}, L^N(\underline{X}^{\ua}_t, \ua_t))\,dt + g(X^{i,\ua}_T, L^N(X^{\ua}_T)) \Big],
\end{equation}
where $\mathcal{A}^N$ is the $N$-fold cartesian product of the set $\mathcal{A}$.
In \cite[Theorem 25]{pontryagin}, conditions on the parameters $b,f$ and $g$ are given to guarantee that, if the problem \eqref{eq:Coop.game} admits a solution $\hat\alpha^N=(\hat\alpha^{1,N},\dots, \hat\alpha^{N,N})$, then for each $i$, and each $t$, the sequence $(\hat\alpha^{i,N}_t)$ converges to $\hat\alpha_t$ in $\mathbb{L}^2$, where $\hat\alpha$ solves the McKean-Vlasov control problem
\begin{equation}
\label{eq:MckV.control}
	\begin{cases}
		\inf_{\alpha \in \mathfrak{A}}E\Big[\int_0^Tf(t, X^{\alpha}_t, \alpha_t, \cL(X^{\alpha}_t,\alpha_t))\,dt + g(X^\alpha_T, \cL(X^\alpha_T)) \Big]\\
		dX^\alpha_t = b(t, X^\alpha_t,\alpha_t,\cL(X^\alpha_t, \alpha_t))\,dt + \sigma\,dW_t,\quad X^\alpha_0 =x.
	\end{cases}
\end{equation}

Since the derivation of the Laplace principle for this game is similar to the non-cooperative case, in order to avoid repetitions we will focus here on the linear quadratic case.
Thus, let us specify the coefficients as follows:
\begin{equation}
\label{eq:LQ-setting-fbg}
	\begin{cases}
		f(t, x, \alpha, \xi) &= Q|x|^2 + \bar Q|\overline x|^2 + R |\alpha|^2 + \bar R |\overline \alpha|^2 + \bar S x \overline \alpha,\\
		b(t, x, \alpha, \xi) &= Ax + \bar A\overline x + B \alpha + \bar B \overline \alpha,\\
		g(x, \mu) &= Q_Tx^2 + \bar Q_T\overline x^2
	\end{cases}
\end{equation}
for some real numbers $A, \bar A, B, \bar B, R, \bar R, Q, \bar Q, Q_T, \bar Q_T, R$ and $\bar R$, and where $\bar x$ and $\bar \alpha$ are the mean of the first and second marginals of $\xi$, respectively.
We assume that $x$ and $\alpha$ have the same dimension denoted $m$.
To state the LDP in this case, let us introduce the following notation:
We consider the functions
\begin{equation*}
	\ell(y, \bar x, \bar y) := -\frac{1}{2 R} \left[ B y  + \bar S  \left( 1 - \frac{\bar R }{R + \bar R} \right) \bar{x} 
			 + \left( \bar B - \frac{\bar R }{R + \bar R}  (B + \bar B) \right) \bar{y}    \right]
\end{equation*}
and
\begin{align*}
	\ell_1(x, y, \xi) &:=  A x + \bar{A}\EE^{\xi_1}[X] + B\ell\Big( y, \EE^{\xi_1}[X], \EE^{\xi_2}[Y] \Big)  + \bar{B}\EE^{\xi_2}\Big[\ell\Big( Y, \EE^{\xi_1}[X], \EE^{\xi_2}[Y] \Big)\Big]\\
		& = Ax + \Big[\bar A - \frac{1}{2R}(SB + \bar S\bar B)\Big(1-\frac{\bar R}{R+\bar R}\Big) \Big]\EE^{\xi_1}[X] \\
		&\quad-\frac{B^2}{2R}y - \frac{1}{2R}\Big[\bar BB + (B+\bar B)\Big(\bar B - \frac{\bar R}{R + \bar R}(B + \bar B) \Big) \Big]\EE^{\xi_2}[Y].
\end{align*}
where $\xi_i$ is the $i^{th}$ marginal of $\xi$, and $\EE^{\xi_1}[X]$ and $\EE^{\xi_2}[Y]$ are the means of $\xi_1$ and $\xi_2$, respectively.
We also consider
\begin{equation*}
	\ell_2(x, y, \xi) = 2 Q x + \Big[ 2 \bar{Q}  -\frac{\bar{S}^2}{2 (R + \bar R)} \Big]\EE^{\xi_1}[X]   + A y + \Big[ \bar{A} -\frac{\bar{S} (B + \bar B)}{2 (R + \bar R)} \Big]\EE^{\xi_2}[Y].
\end{equation*}
With these notation, consider the PDE
	\begin{equation}
	\label{eq:pde.LQ.MFC}
		\begin{cases}
		\partial_tV(t,x,\mu) + \partial_xV(t,x,\mu)\ell_1(x, V(t,x,\mu), \xi) +  \ell_2(x, V(t,x,\mu),\xi)\\
		\quad +\int_{\mathbb{R}^m}\partial_\mu V(t,x,\mu)\cdot \ell_1(y, V(t,x,\mu),\xi)d\mu(y)\\
		 +\frac12tr\Big[\partial_{xx}V(t,x,\mu)\sigma(x)\sigma'(x) + \int_{\mathbb{R}^m}\partial_y\partial_\mu V(t,x,\mu)(y)\sigma(x)\sigma'(x) \,d\mu(y) \Big] = 0\\
			\quad 
			V(T,x,\mu) = 2Q_Tx + 2\bar Q_T \int_{\mathbb{R}^m}x\mu(dx)
		\end{cases}
	\end{equation}
	with $\xi = \text{law}(\chi, V(t, \chi, \mu))$ when $\chi\sim \mu$.
The main result of this section is the following:
\begin{theorem}
\label{thm:LDP.MFC}
	Assume that $R\neq 0$, $R + \bar R\neq 0$.
	If the $N$-player problem \eqref{eq:Coop.game} admits an optimal control $(\hat\alpha^{i,N})_{i=1,\dots,N}$, then, there is $\delta>0$ such that if $T\le \delta$, the sequence $(L^N(\underline{X}^{\hat\ua}))$ satisfies the LDP with rate function
	\begin{equation}
	\label{eq:rate.X1}
		\mathcal{J}(\theta) := \inf_{u\in \mathcal{U}:\, \mathrm{law}(X^u)=\theta}\EE\bigg[\frac12\int_0^T|u_t|^2\,dt\bigg],
	\end{equation}
	where $dX^u_t = \overline \ell_1(t, X^u_t, \cL(X^u_t)) + \sigma u_t\,dt + \sigma\,dW_t$ and $\overline\ell_1$ is the function defined by 
	\begin{equation}
	\label{eq:def.l1}
		\overline \ell_1(t,x,\xi) := \ell_1\Big(x,V(t,x,\xi_1), \cL(\chi, V(t, \chi, \xi_1))\Big)\quad \text{with}\quad \chi\sim \xi_1 \
	\end{equation}
	and $\xi_1$ the first marginal of $\xi$.
\end{theorem}
\begin{proof}
	Let us put
	$$
		\overline x^N := \frac{1}{N} \sum_{j=1}^N x^j, 
		\quad \text{and} \quad
		\overline \alpha^N := \frac{1}{N} \sum_{j=1}^N \alpha^j.
	$$
	It was shown in \cite[Section 5.3.2]{pontryagin} that, if there is an optimal control $(\hat\alpha^{1,N},\dots, \hat\alpha^{N,N})$, then it satisfies
	\begin{align}
	\label{eq:LQMKV-opt-i-expli}
		\hat \alpha^{i,N}_t
			&= 
			-\frac{1}{2 R} \left[ B Y^{i}_t  + \bar S  \left( 1 - \frac{\bar R }{R + \bar R} \right) \overline{X_t}^N 
			 + \left( \bar B - \frac{\bar R }{R + \bar R}  (B + \bar B) \right) \overline{Y_t}^{N}    \right]\\
			\notag
			& = \ell(Y_t^i, \overline X_t^N, \overline Y_t^N)
	\end{align}
	for every $i=1,\dots, N$,
	where the processes $(X^{i,N}, Y^{i,N}, Z^{i,j,N})$ satisfy the FBSDE system
	\begin{empheq}[left=\empheqlbrace]{align*}
			d X^{i}_t &= \left[ A X^{i}_t + \bar{A} \overline{X_t}^N + B \hat \alpha^{i}_t + \bar{B} \overline{\hat \alpha_t}^N  \right]\,dt + \sigma\,dW^i_t
				\\
			d Y^{i}_t 
		 	&= - \left\{ 2 Q X^i_t + \left[ 2 \bar{Q}  -\frac{\bar{S}^2}{2 (R + \bar R)} \right] \overline{X_t}^N 
		 	   + A Y^i_t + \left[ \bar{A} -\frac{\bar{S} (B + \bar B)}{2 (R + \bar R)} \right] \overline{Y_t}^{N}  \right\} dt + \sum_{k=1}^N Z^{i,k}_t dW^{k}_t\\
			&		X^i_0 =x,\quad Y^i_T = 2Q_TX_T^i + 2\bar Q_T\overline{X}^N_T.
	\end{empheq}
	This is essentially a consequence of Pontryagin's maximum principle.
	On the other hand, using again the maximum principle, the optimal control $\hat\alpha$ of the McKean--Vlasov control problem satisfies
	\begin{equation*}
		\hat\alpha_t = -\frac{1}{2 R} \left[ B Y_t  + \bar S  \left( 1 - \frac{\bar R }{R + \bar R} \right)\EE[X_t]
		 + \left( \bar B - \frac{\bar R }{R + \bar R}  (B + \bar B) \right)\EE[Y_t]  \right]
	\end{equation*}  
	where $(X,Y,Z)$ solves the McKean-Vlasov FBSDE 
	\begin{empheq}[left=\empheqlbrace]{align}
		\notag
		d X_t &= \Big\{ A X_t + \bar{A}\EE[X_t]+ B \hat \alpha_t + \bar{B} \EE[\hat\alpha_t]  \Big\}\,dt + \sigma\,dW_t
		\\\label{eq:fbsde.MFC}
		d Y_t &= - \left\{ 2 Q X_t + \Big[ 2 \bar{Q}  -\frac{\bar{S}^2}{2 (R + \bar R)} \Big]\EE[X_t]
		   + A Y_t + \Big[ \bar{A} -\frac{\bar{S} (B + \bar B)}{2 (R + \bar R)} \Big]\EE[Y_t] \right\} dt + Z_t dW_t\\\notag
		&		X_0=x,\quad Y_T = 2Q_TX_T + 2\bar Q_T \EE[X_T],
	\end{empheq}
	see \cite{pontryagin}.
	This is a fully coupled McKean--Vlasov FBSDE system with linear coefficients.
	The existence of a unique solution of this equation is guaranteed e.g. by \cite[Theorem 4.2]{MR3752669} if $T$ is small enough. 
	Furthermore, 
	it follows by \cite[Theorem 2.7, Proposition 5.2]{ChassagneuxCrisanDelarue_Master} that if $T$ is sufficiently small, then the solution $V$ of Equation \ref{eq:pde.LQ.MFC} exists, is Lipschitz--continuous in $(x,\mu)$ and it satisfies $V(t,\chi,\mu) = Y_t^{t,\chi,\mu}$, which is the solution of \eqref{eq:fbsde.MFC} such that $X_t = \chi$ and $\cL(\chi) =\mu$. 
	Therefore, the process $X$ satisfies the (decoupled) equation
	\begin{align*}
		dX_t &= \ell_1\big(X_t, V(t, X_t, \cL(X_t)), \cL(X_t, V(t, X_t, \cL(X_t))) \big)\,dt +\sigma\,dW_t\\
		&= \overline\ell_1(t,X_t, \cL(X_t)) + \sigma\,dW_t.
	\end{align*}
	Now, consider the interacting particle system
	\begin{equation*}
		d\mathcal{X}^{i,N}_t =  \overline{\ell}_1\big(\mathcal{X}_t^{i,N},L^N(\underline{\mathcal{X}}_t)) \big) + \sigma\,d W_t^i
	\end{equation*}
	which is well defined, since $\overline{\ell}_1$ is Lipschitz--continuous, as the composition of two Lipschitz--continuous functions.
	It follows by \cite[Theorem 3.1]{Bud-Dup-Fish} that the sequence $(L^N(\underline{\mathcal{X}}))_N$ satisfies the LDP with rate function $\mathcal{J}$.
	It is checked exactly as in the proof of Theorem \ref{thm:LDP} that $\mathcal{J}$ has compact sublevel sets.
	Thus, it remains to show that $X^{i,N}$ and $\mathcal{X}^{i,N}$ are exponentially close.
	It follows from a similar argument as in Lemma \ref{lem:bound_Y} that
	\begin{equation}
	\label{eq:boundx.y.mckv-con}
		\sup_N\Big\|\sum_{i=1}^N|X^{i,N} - \cX^{i,N}|\Big\|_{\mathcal{S}^{\infty}(\mathbb{R}^\ell)} + \sup_N\Big\|\sum_{i=1}^N|Y^{i,i}-V(\cdot,\cX^{i,N},L^N(\underline\cX))|\Big\|_{\mathcal{S}^{\infty}(\mathbb{R}^\ell)}< \infty.
	\end{equation}
	Subsequently using Chebyshev's inequality and \eqref{eq:boundx.y.mckv-con}, we have
	 \begin{align*}
	 	\PP\Big(\sup_{t\in [0,T]}\cW_2(L^N(\uX_t), L^N(\underline\cX_t)) >\varepsilon\Big) &\le \PP\Big( \Big\{\frac1N\sup_{t\in [0,T]}\sum_{i=1}^N|X^{i,N}_t -  \cX^{i}_t|^2\Big\}^{1/2} \ge \varepsilon \Big)\\
	 	&\le \EE\Big[\exp\Big\{ \Big(\sum_{i=1}^N\|X^{\underline{\alpha},i} -  \cX^{i}\|^2_\infty \Big)\Big\} \Big]e^{-\varepsilon^2N^2}\\
	 	&\le C e^{-\varepsilon^2 N^2} 
	\end{align*}
	from which we deduce that
	\begin{equation*}
		\lim_{N\to \infty}\frac1N\log \PP\Big(\sup_{t \in [0,T]}\cW_1(L^N(\underline{X}_t), L^N(\underline{\mathcal{X}}_t)) \ge \varepsilon\Big) = -\infty.
	\end{equation*}
	Therefore, it follows by \cite[Theorem 4.2.13]{Dembo-Zeitouni}  that the sequence $(L^N(\underline{X}))$ satisfies the LDP with rate function $\mathcal{J}$.
\end{proof}
We conclude the paper with a remark about the small time assumption.
\begin{remark}
	Observe that the smallness assumption made on $T$ in Theorem \ref{thm:LDP.MFC} is needed only to guarantee existence of a classical solution to the PDE \eqref{eq:pde.LQ.MFC}.
	By \cite[Proposition 5.2]{ChassagneuxCrisanDelarue_Master}, this PDE admits a (Lipschitz-continuous) solution on $[0,T]$ for all $T>0$ provided that the FBSDE solution satisfies 
	\begin{equation}
		\EE\big[|Y^{t,\chi,\cL(\chi)}_t - Y_t^{t,\chi',\cL(\chi')} |^2 \big]^{1/2} \le C\EE[|\chi - \chi'|^2]^{1/2}
	\end{equation}
	for every $\chi,\chi' \in L^2(\Omega,\mathcal{F}_t,\PP)$, and for some constant $C>0$ that does not depend on $t$, $\chi$ and $\chi'$.
	This property has been established for instance in \cite[Lemma 5.6]{Car-Del15}, \cite{Ben-Yam-Zhang15} or \cite[Corollary 2.4]{Rei-Sto-Zha2020}, the latter reference assuming monotonicity properties such as those of condition \ref{a8}.
\end{remark}


\bibliographystyle{abbrvnat}
\bibliography{references-Concen_RM}

\vspace{.3cm}

\noindent Peng Luo: School of Mathematical Sciences, Shanghai Jiao Tong University, Shanghai 200240, China.
peng.luo@sjtu.edu.cn\\
Financial support from the National Natural Science Foundation of China (Grant No. 12101400) is gratefully acknowledged. 
  \vspace{.2cm}

\noindent Ludovic Tangpi: Department of Operations Research and Financial Engineering, Princeton University, Princeton, 08540, NJ; USA. ludovic.tangpi@princeton.edu\\
Financial support by NSF grant DMS-2005832 is gratefully acknowledged.
\end{document}